\documentclass{amsproc}
\usepackage{amsfonts,amscd,amssymb,amsmath,amsthm}
\usepackage{graphicx,caption,subcaption,mathrsfs,appendix,centernot}
\usepackage{xparse,tikz}
\usepackage{epstopdf}
\usepackage[T1]{fontenc}

\usepackage{ctable,color}
\usetikzlibrary{arrows,chains,matrix,positioning,scopes,patterns}

\numberwithin{equation}{section}

\begin{document}

\newtheorem{theorem}{Theorem}[section]
\newtheorem{lemma}[theorem]{Lemma}
\newtheorem{lm}[theorem]{Lemma}
\newtheorem{corollary}[theorem]{Corollary}
\newtheorem{conjecture}[theorem]{Conjecture}
\newtheorem{cor}[theorem]{Corollary}
\newtheorem{proposition}[theorem]{Proposition}
\newtheorem{prop}[theorem]{Proposition}
\theoremstyle{definition}
\newtheorem{definition}[theorem]{Definition}
\newtheorem{example}[theorem]{Example}
\newtheorem{claim}[theorem]{Claim}
\newtheorem{remark}[theorem]{Remark}

%%%%%%%%%%%%%%%%%%%%%%%%%%%%%%%%%%%%%%%%%%
%%% General macros
%%%%%%%%%%%%%%%%%%%%%%%%%%%%%%%%%%%%%%%%%%

\newcommand{\R}{\mathbb{R}}
\newcommand{\C}{\mathbb{C}}
\newcommand{\Z}{\mathbb{Z}}
\newcommand{\Q}{\mathbb{Q}}
\newcommand{\E}{\mathbb E}
\newcommand{\N}{\mathbb N}
\newcommand{\X}{\mathbf X}

\newcommand{\del}{\ensuremath{\partial}}
\newcommand{\Def}{\ensuremath{:=}}

\newcommand{\TODO}[1]{\textbf{TODO:}{#1}}

%%%%%%%%%%%%%%%%%%%%%%%%%%%%%%%%%%%%%%%%%%
%%% Probability Macros
%%%%%%%%%%%%%%%%%%%%%%%%%%%%%%%%%%%%%%%%%%

\renewcommand{\Pr}{\mathbb{P}}
\newcommand{\as}{\text{a.s.}}
\newcommand{\Prob}{\Pr}
\newcommand{\Exp}{\mathbb{E}}
\newcommand{\expect}{\Exp}
\newcommand{\1}{\mathbf{1}}
\newcommand{\prob}{\Pr}
\newcommand{\pr}{\Pr}
\newcommand{\filt}{\mathscr{F}}
\DeclareDocumentCommand \one { o }
{%
\IfNoValueTF {#1}
{\mathbf{1}  }
{\mathbf{1}\left\{ {#1} \right\} }%
}
\newcommand{\Bernoulli}{\operatorname{Bernoulli}}
\newcommand{\Binomial}{\operatorname{Binom}}
\newcommand{\Beta}{\operatorname{Beta}}
\newcommand{\Binom}{\Binomial}
\newcommand{\Poisson}{\operatorname{Poisson}}
\newcommand{\Exponential}{\operatorname{Exp}}
\newcommand{\Unif}{\operatorname{Unif}}
\newcommand{\lawequals}{\overset{\mathcal{L}}{=} }

%%%%%%%%%%%%%%%%%%%%%%%%%%%%%%%%%%%%%%%%%%
%%% Random Graph/Complex Macros
%%%%%%%%%%%%%%%%%%%%%%%%%%%%%%%%%%%%%%%%%%

\newcommand{\Deg}{\operatorname{deg}}
\DeclareDocumentCommand \vso { o }
{%
\IfNoValueTF {#1}
{\mathcal{V}  }
{\mathcal{V}\left( {#1} \right) }%
}

\DeclareDocumentCommand \deg { O{ } }
{ \operatorname{deg}_{ #1 }}
\newcommand{\oneE}[2]{\mathbf{1}_{#1 \leftrightarrow #2}}
\newcommand{\ebetween}[2]{{#1} \leftrightarrow {#2}}
\newcommand{\noebetween}[2]{{#1} \centernot{\leftrightarrow} {#2}}
\newcommand{\Gap}{\ensuremath{\tilde \lambda_2 \vee |\tilde \lambda_n|}}
\newcommand{\dset}[2]{\ensuremath{ e({#1},{#2})}}

%%%%%%%%%%%%%%%%%%%%%%%%%%%%%%%%%%%%%%%%%%
%%% Paper-Specific Macros
%%%%%%%%%%%%%%%%%%%%%%%%%%%%%%%%%%%%%%%%%%

\newcommand{\Htwo}{ \mathbb{H}}
\newcommand{\Hd}{ \mathbb{H}^d}
\newcommand{\Etwo}{ \mathbb{E}}
\newcommand{\GWT}{ \mathcal{T}}
\newcommand{\HB}{ B_{\Htwo}}
\newcommand{\MB}{ B_{\mathbb{M}}}
\newcommand{\LB}{ B_{\mathscr{L}}}
\newcommand{\dH}{ d_{\Htwo}}
\newcommand{\dM}{ d_{\mathbb{M}}}
\newcommand{\dL}{ d_{\mathscr{L}}}
\newcommand{\PV}{ \mathcal{V}}
\newcommand{\VD}{ \mathscr{G}^\lambda }
\newcommand{\diamM}{ \operatorname{diam}_{\mathbb{M}} }
\newcommand{\diamH}{ \operatorname{diam}_{\Htwo} }
\newcommand{\dE}{d_{\Etwo}}
\newcommand{\Lattice}{ \Gamma }
\DeclareDocumentCommand \Vol { O{G} }{ \operatorname{Vol}_{#1}}
\newcommand{\VolH}{ \operatorname{Vol}_{\Htwo}}
\newcommand{\VolE}{\operatorname{Vol}_{\Etwo}}
\newcommand{\VolV}{\operatorname{Vol}_{\HPV}}
\newcommand{\VolD}{\operatorname{Vol}_{\HPD}}
\newcommand{\VolM}{ \operatorname{Vol}_{\mathbb{M}}}
\newcommand{\convH}{ \operatorname{conv}_{\Htwo}}
\newcommand{\HPV}{ \mathscr{V}^\lambda}
\newcommand{\HPD}{ \mathscr{D}^\lambda}

\newcommand{\Del}{ \operatorname{Del}}
\newcommand{\Star}{ \operatorname{St}}
\newcommand{\GF}{\mathcal{G}_f}
\newcommand{\DeltaH}{\Delta_{\Htwo}}
\newcommand{\DeltaE}{\Delta_{\Etwo}}
\newcommand{\angleH}{\angle_{\Htwo}}
\newcommand{\angleE}{\angle_{\Etwo}}

%circumdisk
\newcommand{\CDisc}{ \operatorname{CD}_{\Htwo}}
\newcommand{\CCtr}{ \operatorname{CC}_{\Htwo}}

%Poisson
\newcommand{\PPP}{ \Pi^\lambda }

\DeclareDocumentCommand \TFX { O{i} O{\pi_{\X}} O{f}  }
{ \Delta_{ {#2}, {#3} }({#1})}

\DeclareDocumentCommand \Filt { O{n} }
{ \mathscr{F}_{{#1}} }

\DeclareDocumentCommand \isol { O{i} O{S} }
{ \Delta_{#1} \left( {#2} \right) }

\newcommand{\aec}{i^*}
\DeclareDocumentCommand \islands {  O{i} }
{ \operatorname{IS}_{ {#1 }} }

\newcommand{\BHF}{ \mathcal{B}}

\title[Distributional lattices]{Distributional Lattices on Riemannian symmetric spaces}

\author{Elliot Paquette}
\address{Department of Mathematics, The Ohio State University}
\email{paquette.30@osu.edu}



%    author one information

%    The 2010 edition of the Mathematics Subject Classification is
%    now available.  If you are citing a classification from the
%    new scheme, use the following input coding instead.
%\subjclass[2010]{Primary }

\date{\today}

\begin{abstract}

  A Riemannian symmetric space is a Riemannian manifold in which it is possible to reflect all geodesics through a point by an isometry of the space.  On such spaces, we introduce the notion of a distributional lattice, generalizing the notion of lattice.  Distributional lattices exist in any Riemannian symmetric space: the Voronoi tessellation of a stationary Poisson point process is an example.  We show that for an appropriate notion of amenability, the amenability of a distributional lattice is equivalent to the amenability of the ambient space.  Using this equivalence, we show that the simple random walk on any nonamenable distributional lattice has positive embedded speed.  For nonpositively curved, simply connected spaces, we show that the simple random walk on a Poisson--Voronoi tessellation has positive graph speed by developing some additional structure for Poisson--Voronoi tessellations.
\end{abstract}

\maketitle

\section{Introduction}

\subsection*{Riemannian symmetric spaces}

A Riemannian symmetric space $\mathbb{M}$ is a connected Riemannian manifold where at each point $p,$ there is an isometry $\sigma_p$ of $\mathbb{M}$ that fixes $p$ and whose differential at $p$ is multiplication by $-1.$ 

Riemannian symmetric spaces provide many excellent examples of nonpositively and positively curved spaces, which include:
\begin{enumerate}
  \item the Euclidean spaces,
  \item the spheres $\mathbb{S}^n$ in Euclidean space,
  \item the real, complex, and quaternionic hyperbolic spaces ($\Htwo^d, \C\Htwo^d, \Htwo\Htwo^d$ respectively, see \cite[Chapter 10]{BridsonHaefliger} for a comprehensive treatment) of any dimension,
  \item $\text{SL}_d(\R)/\text{SO}_d(\R)$ for any $d \geq 1,$ which can be identified with positive definite matrices modulo scalars,
  \item Riemannian products of any of the above examples, such as $\Htwo^d \times \R^k$ or $\Htwo^2 \times \Htwo.$
\end{enumerate}
%$\Htwo^d$, $\Htwo^d \times \R^k$, $\Htwo \times \Htwo,$ $\text{SL}_d(\R)/\text{SO}_d(\R)$ (which can be identified with positive definite matrices), and many others.
All of the examples listed except for the spheres are nonpositively curved.
Additionally, Riemannian symmetric spaces decompose nicely: any simply connected Riemannian symmetric space decomposes as a Riemannian direct product $\mathbb{M}_1 \times \R^d \times \mathbb{M}_2$ where $\mathbb{M}_1$ is nonpositively curved and $\mathbb{M}_2$ is compact (see \cite[V,Proposition 4.2]{Helgason}).  
%Hence, the projection from $\mathbb{M}$ onto $\mathbb{M}_1$ is a quasi-isometry.%, so restricting to nonpositively curved Riemannian symmetric space is natural from the point of view of understanding the large scale geometry of Poisson Voronoi graphs.

When a Riemannian symmetric space has no factors of $\R^k$ and no compact factors in its de Rham decomposition (its decomposition as a Riemannian product into irreducible factors), it is called a Riemannian symmetric space of noncompact type.  Such spaces can be identified as quotient spaces $G/K$ where $G$ is a semisimple Lie group with trivial center and $K$ is a maximal compact subgroup. 

\subsection*{Lattices and Voronoi tessellations}

All Riemannian symmetric spaces $\mathbb{M}$ are diffeomorphic to quotient spaces $G/K$ where $G$ is the isometry group of $\mathbb{M}$ and $K$ is a stabilizer of some point (see \cite[V, Theorem 3.3]{Helgason}).  Going forward, we let $o \in \mathbb{M}$ be an arbitrary base point, and let $K$ refer to the stabilizer of $o$ in the isometry group of $\mathbb{M}.$  The group $G$ has a locally compact Lie group structure and so has a (left) Haar measure. 

Generally, a discrete subgroup $\Gamma$ of a Lie group is called a \emph{lattice} if there is a measurable set of coset representatives of $\Gamma \setminus G$ so that $\Gamma \setminus G$ has finite Haar measure.  When $G$ is the isometry group of $\mathbb{M},$ one can identify an equivalent condition to $\Gamma$ in terms of Voronoi tilings.  Every simply connected Riemannian symmetric space has many lattices \cite{Borel}.  In particular, each Riemannian symmetric space of noncompact type has many lattices.
%While not all lie groups have lattices, every semisimple lie group has many lattices (\cite{Borel}) and hence every isometry group of a Riemannian symmetric space has many lattices.

The orbit $\Gamma \cdot o$ forms a closed discrete subset of $\mathbb{M}.$  Hence, it is possible to define the \emph{Voronoi tessellation} of $\mathbb{M}$ with \emph{nuclei} $\Gamma \cdot o.$  In general, for a closed discrete set $S \in \mathbb{M},$ the Voronoi tessellation with nuclei $S$ is a decomposition of $\mathbb{M}$ into \emph{cells}, which for a given point $x \in S$ is defined by
\[
  \vso(x ; S) = \left\{ y \in \mathbb{M} : \dM(x,y) = \min_{s \in S} \dM(s,y) \right\}.
\]
Note that $\Gamma$ permutes the voronoi cells $\{\vso( \gamma \cdot o; \Gamma \cdot o) : \gamma \in \Gamma\},$ and hence all of them have equal volume. 
On account of the compactness of the stabilizer $K,$ it is readily checked that the finiteness of the Riemannian volume of $\vso(o)$ is equivalent to the finiteness of $\Gamma \setminus G$ under the Haar measure of $G.$

\subsection*{Stationary point processes}

A (simple) point process $\mathcal{P}$ on a complete separable metric space $\mathcal{X}$ is a probability measure on countable subsets of $\mathcal{X}$ which have finite intersection with any compact set.  Often these are considered as random locally--finite, integer--valued measures by summing point masses at each of these points.  In this way, we can view $\mathcal{P}$ as a random element of $\mathcal{M}_\mathcal{X},$ the space of locally finite measures on $\mathcal{X}.$  When $\mathcal{X}$ is a homogeneous space, we say that $\mathcal{P}$ is a stationary point process if $\mathcal{P} \lawequals \tau(\mathcal{P})$ for any isometry $\tau$ of $\mathcal{X}.$

We will only consider point processes that have finite intensity, which is to say that for any compact $A \subset \mathcal{X}$
\[
  \Lambda(A) := \Exp | A \cap \mathcal{P} |  < \infty.
\]
Note that for stationary point processes on a Riemannian homogeneous space $\mathbb{M}$, the \emph{intensity} measure $\Lambda$ will then be an invariant Radon measure, and hence it is a multiple of the Riemannian volume measure by the uniqueness of Haar measure.  We will let $\lambda$ denote this multiple.

The Palm process $\mathcal{P}_o,$ which is again a point process, has distribution which can be considered as the distribution of $\mathcal{P}$ conditioned to have a point at $o.$  For a general discussion of the theory, see \cite[Ch 6]{Kallenberg} or \cite[II.13]{VereJones}.  More formally, it is defined as a family of point processes $\left\{ \mathcal{P}_o, o\in \mathcal{X} \right\}$ so that for any nonnegative measurable function $f : \mathcal{X} \times \mathcal{M}_\mathcal{X} \to \R$
\[
  \Exp\left[
    \int_{\mathcal{X}} f(o, \Xi) \Xi(do)
  \right]
  =
 \Exp\left[
    \int_{\mathcal{X}} f(o, \Xi_o) \Lambda(do)
  \right].
\]
Note that $\left\{ \mathcal{P}_o, o\in \mathcal{X} \right\}$ is only well--defined up to $\Lambda$--null sets.
For a stationary point process, one version of the Palm process is given by $\left\{ \tau_{o,x}(\mathcal{P}_o), x \in \mathcal{X} \right\}$ where $\tau_{o,x}$ is an isometry of the space mapping $o$ to $x.$  In particular, for a stationary point process, the palm processes $\mathcal{P}_x$ are meaningful for all points $x \in \mathcal{X}.$ 

\begin{figure}
  \begin{subfigure}[t]{0.45\textwidth}
    \includegraphics[width=\textwidth]{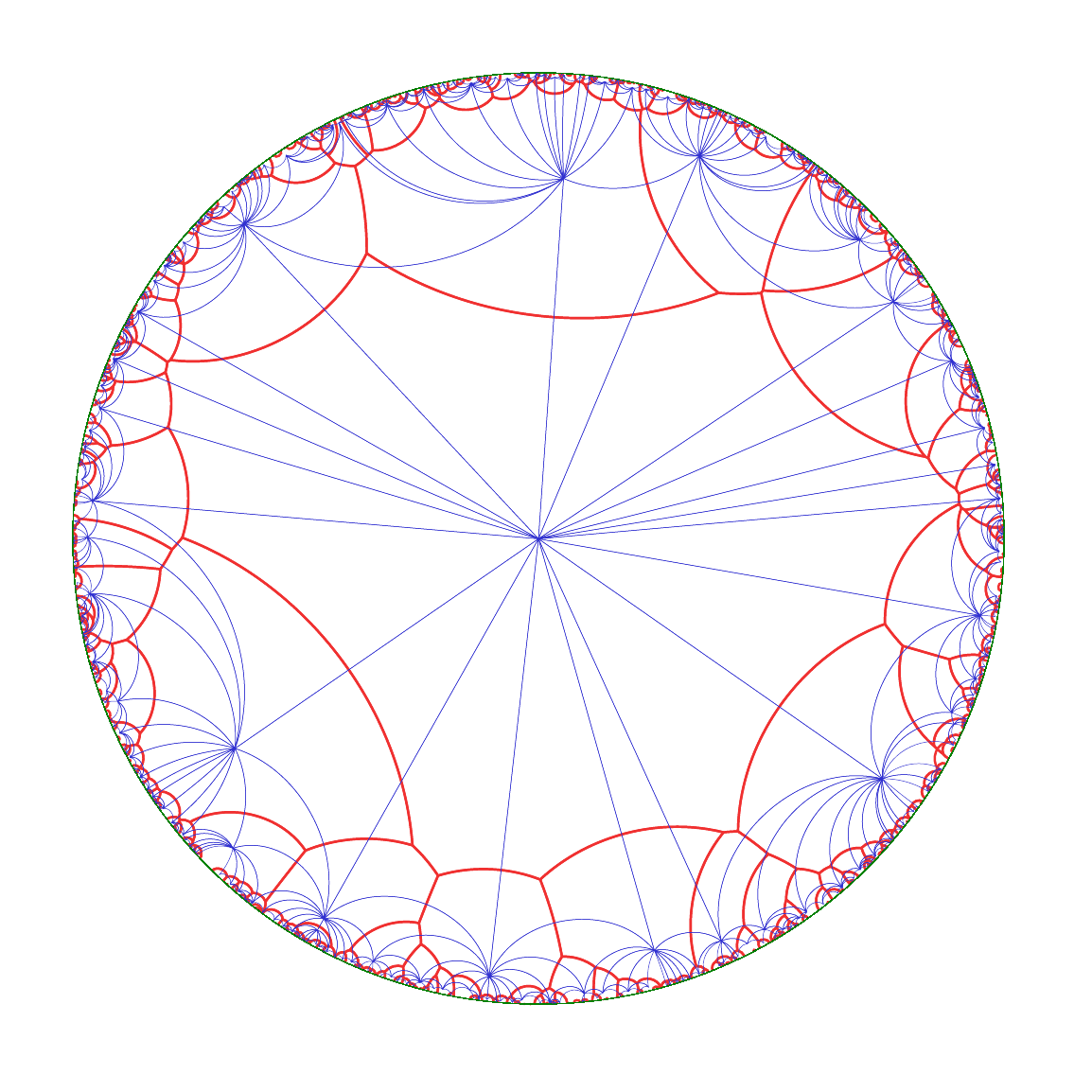}
    \caption{The hyperbolic GAF (see \cite[Ch. 5]{HHPV}), approximated by the roots inside the disk of a Kac polynomial of degree $1000,$ of which there are in expectation $500.$}
  \end{subfigure}
  \begin{subfigure}[t]{0.45\textwidth}
    \includegraphics[width=\textwidth]{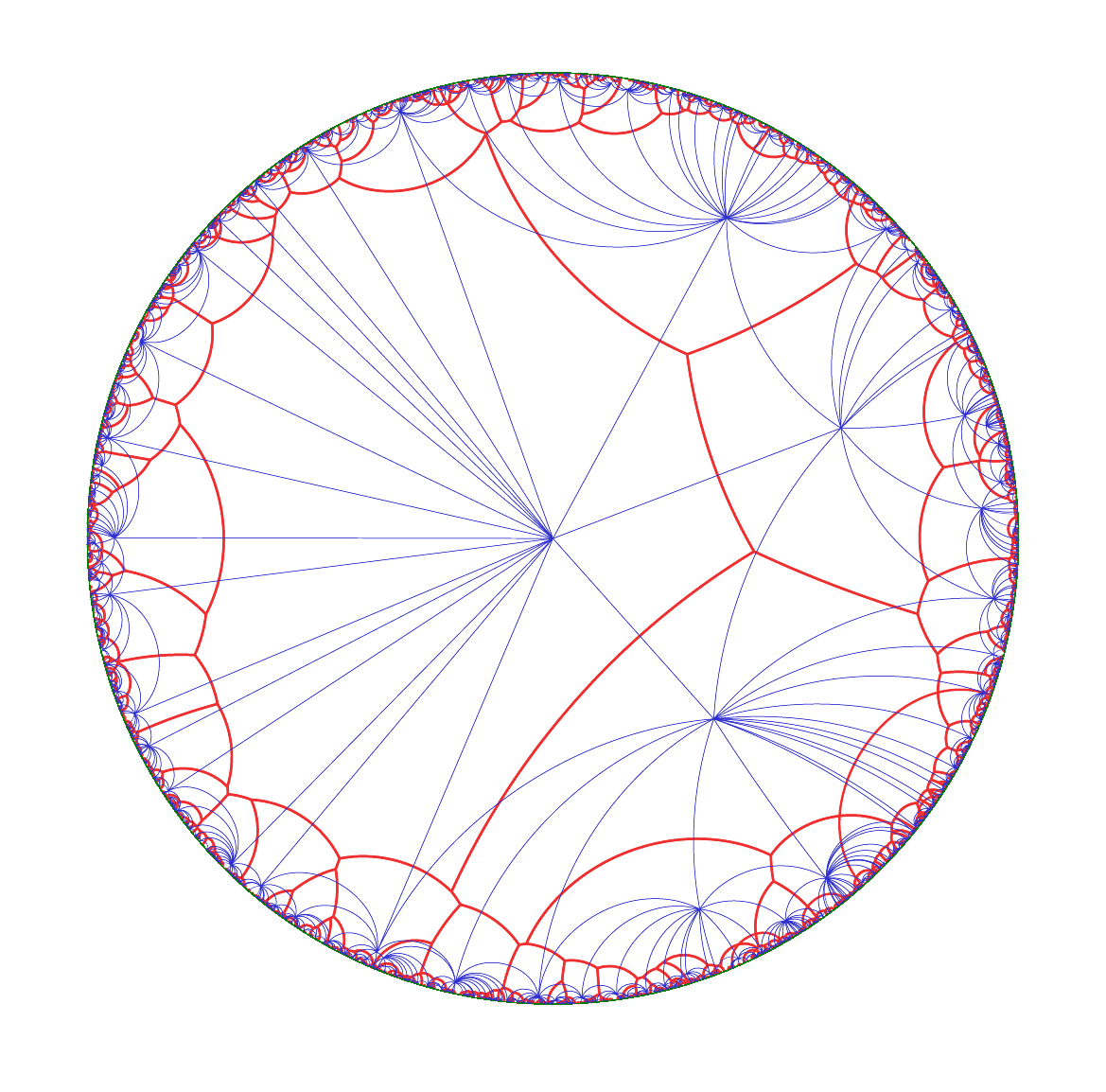}
    \caption{A Poisson--Voronoi tiling with intensity matching the hyperbolic GAF, truncated in a disk of radius $r$ where $r$ is chosen to have $500$ points in expectation (matching the Kac polynomial).} 
  \end{subfigure}

  \caption{These are simulations depicting the hyperbolic tessellations for the palm processes of two different stationary point processes.  They are chosen to have matching intensities.  The GAF is statistically more rgular than the Poisson process.  The points appear as the vertices of the triangulation (the blue graph, where color is available), and the dual tessellation is the Voronoi tessellation. 
  }
  \label{fig:simulation}
\end{figure}

\begin{example}[Poisson--processes]

  While some of the results in this paper are applicable to general palm processes of stationary point processes, we are able to show much more for the case that $\mathcal{P}$ is an stationary Poisson process.  The \emph{Poisson--process with intensity $\lambda \cdot dV$} where $dV$ is the Riemannian volume measure and $\lambda > 0$ is a parameter.  Such a process can be uniquely determined by the condition that for any disjoint collection of compact sets $A_1,A_2,\dots,A_k,$ the cardinalities $\left\{ |\mathcal{P} \cap A_i| \right\}_1^k$ are independent Poisson--distributed variables with means $\lambda V(A_i).$  One can further see that conditional on the cardinality of $|\mathcal{P} \cap A|,$ the points of $\mathcal{P} \cap A$ are independent and uniformly distributed over $A.$

The Palm process $\mathcal{P}_o$ has the attractive feature that 
$\mathcal{P}_o \lawequals \{o\} \cup \mathcal{P},$
which is to say conditioning the Poisson process to have a point at the origin leaves the distribution unchanged on $\mathcal{X} \setminus \{ o\}.$  

\end{example}

\begin{example}[Translated--lattices]
  
  The lattice $\Gamma \cdot o$ forms another point process, albeit completely deterministic.  If one randomizes the process by taking $g \Gamma \cdot o$ where $g$ is distributed according to Haar measure on $\Gamma \setminus G$ then the resulting point process is invariant under all isometries by virtue of the $G$ invariance of the Haar measure on $\Gamma \setminus G.$  The palm process that results from conditioning $o$ to be at the origin has the distribution of $k \Gamma \cdot o,$ where $k$ is distributed according to Haar measure on $K,$ the stabilizer of $o.$ 

\end{example}

\begin{example}[Determinantal--processes]

  A more exotic class of examples are the determinantal processes.  In two--dimensions there are relatively well--studied examples of stationary determinantal point processes for all of the constant curvature spaces.  These examples have the additional interpretation as being zero sets of Gaussian random analytic functions.  See \cite[Chapter 2]{HHPV} for more details.  Some properties of the Voronoi diagrams have been studied for Ginibre process \cite{Goldman}.

  These processes exist in much greater generality than suggested.  Many stationary determinantal processes exist in any Riemannian symmetric space.  In effect, it is possible to define a determinantal process for any isometry--invariant closed subspace $S$ of $\text{L}^2(dV)$ for which the evaluation maps $f \mapsto f(z)$ are bounded linear functionals (for example spaces of certain harmonic functions).  On such spaces, one can find a reproducing kernel $\mathcal{K}(z,w) = \sum_{k=1}^\infty \psi_k(z)\psi_k(w),$ for a complete orthonormal system $\left\{ \psi_k \right\}_1^\infty$ of $S.$  The intensity of points will be given $\mathcal{K}(x,x)dV(x).$  By invariance, it follows that $\mathcal{K}(x,x)$ is constant.  This construction leads to the canonical examples discussed in the prior paragraph.  A particularly nice example is given by the Bergman kernel associated to any bounded symmetric domain in $\C^n$ (see \cite[VII.3,7]{Helgason} for details).  The hyperbolic GAF depicted in Figure~\ref{fig:simulation} is such an example, with the domain being the unit disk inside of $\C.$
\end{example}

\subsection*{Unimodular networks}

We follow the notation and development of this material from \cite{AL}, which should be consulted for more details.
A network is a graph $(V,E)$ in which vertices and edges are marked by some elements of a complete separable metric space, that is to say there is a complete separable metric space $\Xi$ and maps $\psi_V: V \to \Xi$ and $\psi_E : E \to \Xi.$  A network is rooted if there is a designated vertex $o \in V,$ and birooted if there is an ordered pair $(x,y)$ of distinguished vertices.  Let $\mathcal{G}_*$ be the space of isomorphism classes of rooted networks, and let $\mathcal{G}_{**}$ be the isomorphism classes of birooted networks.  A random rooted network is \emph{unimodular} if it obeys the \emph{mass-transport principle}, which is to say for all Borel $f : \mathcal{G}_{**} \to [0,\infty],$
\begin{equation}
  \Exp\left[ 
    \sum_{x \in V} f(G,o,x,\psi)
  \right]
  =
  \Exp\left[ 
    \sum_{x \in V} f(G,x,o,\psi)
  \right].
  \label{eq:mtp}
\end{equation}

The Voronoi tessellation on $\mathbb{M}$ with nuclei $\mathcal{P}_o$ naturally gives rise to a unimodular network.  
The graph $G$ is the \emph{Delaunay graph} on vertices $\mathcal{P}_o.$  
Two vertices $x$ and $y$ are connected by an edge if and only if $\vso(x) \cap \vso(y)$ has codimension-1.  Equivalently, two vertices $x$ and $y$ are connected if and only if there is a ball $\MB(z,r)$ having the property that $x,y \in \partial \MB(z,r)$ and $\MB(z,r) \cap \mathcal{P}_o$ is empty.  
Under some genericity conditions, which for example are satsified by Poisson points with probability $1,$ this graph is a $1$--skeleton of a triangulation of $\mathbb{M}.$

The graph will be rooted at the vertex at $o.$
The space of marks $\Xi$ will be taken to be the manifold $\mathbb{M}.$  
The vertices will be marked by the locations of the points, and the edge marks will be the midpoints of the $\mathbb{M}$--geodesic between its endpoints.  We let $\psi$ denote this marking.  We refer to this construction of random network as the \emph{embedded Delaunay graph} with nuclei $\mathcal{P}_o.$

\begin{theorem}
  Let $\mathcal{P}_o$ be the palm process of a stationary point process, and let $(G,o,\psi)$ be the embedded Delaunay graph with these nuclei.  Then $(G,o,\psi)$ is a unimodular random network. 
  \label{thm:network}
\end{theorem}
\noindent The proof is given in Section~\ref{sec:dl}

This leads us to the following generalization of lattices.
\begin{definition}
  A \emph{distributional lattice} in a Riemannian symmetric space $\mathbb{M}$ is the palm process $\mathcal{P}_o$ of a stationary point process with the property that $\Exp[\Vol[\mathbb{M}](\vso(o))] < \infty$ and $\Exp[ \Deg_G(o)] <\infty.$
\end{definition}
When a Riemannian symmetric space has a lattice $\mathcal{L},$ it can be made into a distributional lattice by picking a Haar distributed coset.  Also, palm processes of stationary Poisson processes with positive finite intensity are always distributional lattices (see Theorem~\ref{thm:pv}).

\subsection*{Amenability}

A homogeneous Riemannian manifold $G/H$ is amenable if for every compact subset $S \subset G$ there is a sequence of measurable sets $V_n$ with finite volume so that
\[
  \lim_{n \to \infty} \max_{g \in S} \frac{|g V_n \Delta V_n|}{|V_n|} = 0,
\]
with $|\cdot|$ denoting the volume measure.
A locally compact topological group $G$ is amenable if for every compact subset $S \subset G,$ there is a sequence of measurable sets $V_n$ with finite positive Haar measure so that
\[
  \lim_{n \to \infty} \max_{g \in S} \frac{\mu( g V_n \Delta V_n)}{\mu(V_n)} \to 0.
\]
If $G$ is amenable, then as $H$ is necessarily closed (as $G/H$ is a manifold), the amenability of $H$ follows.  On the other hand, if $H$ is amenable (which is always the case for a Riemannian symmetric space, in which case we can take it to be compact), the amenability of $G/H$ implies the amenability of $G.$  In particular, from this we can conclude that all Riemannian symmetric spaces $\mathbb{M}$ of noncompact type are nonamenable, and hence, any Riemannian symmetric space with a de Rham factor of noncompact type is nonamenable as well.

A graph $G$ is (edge) \emph{non}-amenable if 
\begin{equation}
  i^E(G) := \inf_{\substack{V \subset G \\ |V| <\infty}} \frac{|\partial_E V|}{\Vol(V)} > 0.
  \label{eq:ieG}
\end{equation}
where $\partial_E V$ is the set of edges crossing from $V$ to its complement and $\Vol(V_n)$ is the sum of degrees of vertices in $V.$  In a Riemannian symmetric space $G/K$ having lattices, the amenability of the space is equivalent to the existence of a lattice $\mathcal{L} < G$ for which the Delaunay graph with nuclei $\mathcal{L}\cdot o$ is amenable.  In particular the following theorem is classical (see \cite{Gelander} for a discussion).
\begin{theorem}
  In a Riemannian symmetric space $\mathbb{M}$ with distinguished point $o$ for which lattices exist, the following are equivalent.
  \begin{enumerate}
    \item $\mathbb{M}$ is amenable.
    \item For every lattice $\mathcal{L}$ in the isometry group of $\mathbb{M},$ the Delaunay graph with nuclei $\mathcal{L} \cdot o$ is amenable.
    \item There exists a lattice $\mathcal{L}$ in the isometry group of $\mathbb{M}$ so that the Delaunay graph with nuclei $\mathcal{L} \cdot o$ is amenable.
  \end{enumerate}
  \label{thm:la}
\end{theorem}

Many random graphs, such as Galton Watson trees with offspring distribution having positive probability of fewer than $2$ children and Poisson--Delaunay graphs are amenable due to random fluctuations, regardless of the amenability of the underlying space $\mathbb{M}.$ For this reason, one needs a weaker notion of nonamenability (respectively a stronger notion of amenability) to characterize how nonamenability appears in these random graphs.

One such notion is \emph{anchored amenability}.  A rooted graph $(G,o)$ is anchored amenable if there is a sequence of finite sets $V_n$ so that $o \in V_n,$ so that the induced subgraph on $V_n$ is connected, and so that
\[
  \lim_{n \to \infty} \frac{|\partial_E{V_n}|}{\Vol(V_n)} = 0.
\]
It is shown in \cite{BPP} that the Poisson--Delaunay graph in $\Htwo^2$ is anchored nonamenable.  Forthcoming work in \cite{BKP} show that a slightly weaker version of anchored nonamenability holds in $\Htwo^d.$

Another, still weaker notion is \emph{invariant amenability}, introduced in \cite{AL} and used to great effect in the study of unimodular random triangulations in \cite{AHNR}.  For a unimodular random network $(G,o,\psi)$ with marks $\Xi$, a \emph{unimodular percolation} is another unimodular random network $(G,o,\psi \times \omega)$ with an augmented mark space $\Xi \times \left\{ 0,1 \right\}$ which is again unimodular.  One may consider this as a random subnetwork of the original space by taking only those vertices $v$ and edges $e$ so that $\omega(v)=\omega(e)=1,$ and we assume that if an edge is present in the network, then so are both its endpoints.  We let $K_\omega(v)$ denote the connected component of $v$ in this subnetwork.  If $\omega(v) = 0,$ we take $K_\omega(v) = \{v\}$ by convention.

A percolation is \emph{finitary} if all of its clusters are finite almost surely.  A unimodular network is \emph{invariantly amenable} if
\[
  \inf \left\{ \Exp \frac{|\partial_E K_\omega(v)|}{|K_\omega(v)|} : \text{$\omega$ a finitary unimodular percolation}  \right\}  = 0.
\]
This notion allows us to give a precise generalization of Theorem~\ref{thm:la} for stationary point processes. 
\begin{theorem}
  In a Riemannian symmetric space $\mathbb{M}$ with distinguished point $o$, the following are equivalent.
  \begin{enumerate}
    \item $\mathbb{M}$ is amenable.
    \item For every
      distributional lattice $\mathcal{P}_o$
      %palm process $\mathcal{P}_o$ of a stationary point process,
      the embedded Delaunay graph with nuclei $\mathcal{P}_o$ is invariantly amenable.
    \item There exists 
      a distributional lattice $\mathcal{P}_o$
      %a palm process $\mathcal{P}_o$ of a stationary point process
      %with finite volume Voronoi cells 
      so that the embedded Delaunay graph with nuclei $\mathcal{P}_o$ is invariantly amenable.
  \end{enumerate}
  \label{thm:lat}
\end{theorem}
In particular, if $\mathbb{M}$ is nonamenable, every palm process $\mathcal{P}_o$ of a stationary point process with finite volume Voronoi cells produces an invariantly nonamenable Delaunay graph.

\subsection*{Ergodic connections and applications}

There are some immediate consequences that one can draw from the invariant nonamenability of a distributional lattice.  The first conclusion we draw uses the following theorem. 
\begin{theorem}[Theorem 3.2 of \cite{AHNR}]
  Let $(G,o)$ be an invariantly nonamenable unimodular random rooted graph with $\Exp\left[ \deg(o)< \infty \right].$  Then $G$ admits a unimodular percolation $\omega$ so that the open subgraph $G\vert_\omega$ has $i^E(G\vert_\omega) > 0$ (c.f.\,\eqref{eq:ieG}) almost surely and so that every vertex in $G\vert_\omega$ has uniformly bounded degrees in $G.$
  \label{thm:bdedexpander}
\end{theorem}
In the cited theorem, the additional assumption of ergodicity is stated, but by the ergodic decomposition this version follows immediately from that one.

Let $\left( X_j : j=0,1,2,\dots\right)$ be simple random walk on $(G,o),$ the Markov chain on the vertices of $G$ where $X_{j+1}$ is distributed uniformly over the neighbors of $X_j$ for all $j=0,1,\dots.$  We take $X_0 = o.$  
\begin{theorem}[Theorem 4.1 of \cite{AL} or Proposition 2.5 of \cite{BenjaminiCurien}]
 Define the probability measure $\mathbb{Q}$ absolutely continuous with respect to the law of $(G,o,\psi,\left( X_j \right)_1^\infty)$ with Radon-Nikodym derivative $\deg(o)/\Exp \deg(o).$  Fix a a representative set of isometries $G/K,$ and let $\psi_j$ be a translate of $\psi$ by that representative isometry of $\mathbb{M}$ so that $\psi_j(X_j) = \psi(o).$ Under $\mathbb{Q}$ the random, rooted embedded Delaunay graphs $((G,X_j,\psi_j))_1^\infty$ are stationary.  That is to say the shift map $T$ defined by
 \[
   T(((G,X_j,\psi_j))_1^\infty) 
   = 
   ((G,X_j,\psi_j))_2^\infty 
 \]
 is $\mathbb{Q}$--measure preserving 
  \label{thm:bias}
\end{theorem}
Note that by mutual absolute continuity of $\mathbb{Q}$ and the original law, events that hold almost surely under one measure hold almost surely under the other.

\begin{corollary}
  Let $\mathbb{M}$ be nonamenable.
  For any distributional lattice $\mathcal{P}_o \subset \mathbb{M}$ the simple random walk on the embedded Delaunay graph $(G,o,\psi)$ has positive speed in the sense that
  \[
    s=\lim_{n \to \infty} \frac{\dM(X_j,o)}{j}
  \]
  is positive almost surely.  If in addition $\mathbb{M}$ is simply connected and nonpositively curved, then for the Poisson--Delaunay graph
  \[
    s^G =\lim_{n \to \infty} \frac{d_G(X_j,o)}{j}
  \]
  is positive almost surely.
\end{corollary}

This extends the conclusions of \cite{BPP} on random walk on the Poisson-Delaunay graph in the hyperbolic plane, where it was shown that $s >0$ almost surely.  See also \cite{CLP} where the low--intensity asymptotics of the speed are determined for the Poisson--Delaunay graph in $\mathbb{H}^d$.

\begin{proof}
  The proof is identical to \cite[Theorem 1.1]{BPP}, with the random graph $G\vert_\omega$ taking the place of $G \setminus \text{IS}.$  We outline the proof here.
  
  By the subadditive ergodic theorem \cite[Theorem 6.1]{Durrett}, both $s$ and $s^G$ exist almost surely.  Hence it will suffice to show positive $\limsup$ speed.

  Using that the intensity of $\mathcal{P}$ is a multiple of Riemannian volume measure, by appealing to Markov's inequality and Borel-Cantelli, the number of points in an $\mathbb{M}$-ball of radius $R$ has exponential growth in $R$, i.e.
  \[
    |\left\{ x \in \mathcal{P}_o : \dM(x,o) \leq R\right\}| \leq M e^{\alpha R},\quad \text{for all $R > 0$}
  \]
  for some deterministic $\alpha$ and some almost surely finite random variable $M.$

  Let $\left( T_k \right)_1^\infty$ be the times at which $X_j \in G\vert_\omega.$  Then the induced random walk $\left( Y_k = X_{T_k}: k=1,2,\dots \right)$ has a spectral gap, and so its transition probabilities decay exponentially in $k.$  Hence applying Borel--Cantelli, we have that
  \[
    \dM(Y_k,o) \geq \beta k
  \]
  for some $\alpha >0$ and all $k$ sufficiently large.

  We then transfer the result to $X_k$ by invoking the ergodic theorem (under $\mathbb{Q}$) due to which 
  \[
    \lim_{n\to\infty} \frac{1}{n} \sum_{j=1}^n \one\left[ X_k \in G\vert_\omega \right] > 0,
  \]
  which shows that there is at least a subsequence $k_n$ along which 
  \[
    \dM(X_{k_n},o) > \gamma k_n.
  \]
  As the limit $s$ exists almost surely, it follows it is positive almost surely.

  As for the graph speed, we have by Proposition~\ref{prop:expgrowth} that there is a constant $t > 0$ so that for all $r$ sufficiently large
  \[
    \psi(B_{G}(o,r))
    \subset
    \MB(o,tr),
  \]
  and hence the graph speed must also be positive.

\end{proof}

\subsection*{Outline of the paper}

In Section \ref{sec:dl}, we prove Theorem~\ref{thm:la}.
In Section \ref{sec:pv}, we show that the Poisson--Voronoi tessellation is always a distributional lattice.  Some of the facts developed here are also applied in Section \ref{sec:dl}.
In Section \ref{sec:npc}, we prove some additional properties of Poisson--Voronoi tessellations in nonpositively curved spaces.  Specifically, we show that the graph ball $B_G(o,R)$ for large $R$ fits with high probability inside the manifold ball $\MB(o,tR)$ for some $t > 0$ (Proposition~\ref{prop:expgrowth}).  We also show that Poisson--Voronoi tessellations in simply connected, nonpositively curved spaces are \emph{sofic} in the sense that they are local limits of uniformly rooted finite random networks (Proposition~\ref{prop:sofic}).

\subsection*{Acknowledgements}
The author would like to thank Mat\'ias Carrasco Piaggio and Pablo Lessa for many helpful conversations.  Thanks also to Itai Benjamini for helpful comments and inspiring conversations.

\section{Distributional lattices}
\label{sec:dl}
The main goal of this section is the proofs of Theorems~\ref{thm:network} and~\ref{thm:la}.
We begin with the proof of Theorem~\ref{thm:network}.

\begin{proof}[Proof of Theorem~\ref{thm:network}]
  By \cite[Proposition 2.2]{AL}, it suffices to show that for nonnegative $f$ supported on $(G,x,y,\psi)$ with $x \sim y$ that the mass transport principle holds. For any two nonequal points $y,z \in \mathbb{M},$ let $\mathbb{B}_{y,z}$ be the bisector of $y$ and $z,$ i.e.\,the (not necessarily totally geodesic) submanifold of $\mathbb{M}$ consisting of points that are equidistant from $y$ and $z,$ and let $\mathcal{C}_{y,z}$ be the event that $\left\{ y,z \right\} \subset \mathcal{X}$ and $y$ connects to $z$ in ${G}.$  This event 
\[
  \mathcal{C}_{y,z}(\mathcal{P}_o)
= \left\{ 
  \exists~u\in \mathbb{B}_{y,z}~:~\MB(u,\dM(u,y)) \cap \mathcal{P}_o = \emptyset
\right\}.
\]

  For any bijection $\tau : \mathbb{M} \to \mathbb{M},$ let $\tau^* : \sigma(\mathcal{X}) \to \sigma(\mathcal{X})$ be the induced map on events, i.e.\,for measurable $V \subset \mathbb{M}^{\mathbb{N}}$
  \[
    \tau^*(\left\{ \mathcal{X} \in V \right\}) = 
    \left\{ \tau(\mathcal{X}) \in V \right\}.
  \]
  Let $\tau_{y,z}$ denote the isometric involution of $\mathbb{M}$ that reverses the geodesics through the midpoint $m$ between $y$ and $z$ and hence interchanges $y$ and $z.$  As $\mathbb{B}_{y,z}$ is stabilized as a set under $\tau_{y,z}$ it follows that for any $y \in \mathbb{M}$
  \[
    \tau_{o,y}^*(\mathcal{C}_{o,y}(\mathcal{P}_o)) 
    = \mathcal{C}_{o,y}(\mathcal{P}_y).
  \]
  This is because if $\MB(u,r)$ is a ball centered at a point on the bisector $\mathbb{B}_{o,y}$ not intersecting $\mathcal{P}_o$ then $\MB(\tau_{o,y}(u),r)$ is a ball centered at a point on the bisector $\mathbb{B}_{o,y}$ not intersecting $\mathcal{P}_y = \tau_{o,y}(\mathcal{P}_o),$ and vice versa. 

  Using Palm theory,
  \[
    \Exp\left[ 
      \sum_{x \in \mathcal{P}_o} f(G,o,x,\psi)
    \right]
    =
    \int_{\mathbb{M}}
    \Exp_{\mathcal{P}_{o,x}}\left[ 
      f(G,o,x,\psi)\one\left[ \mathcal{C}_{o,x}(\mathcal{P}_{o,x}) \right]
    \right]
    \Lambda_{o}(dx),
  \]
  where $\Lambda_{o}$ is the intensity of the reduced palm process $\mathcal{P}_o \setminus \left\{ o \right\}$ and $\mathcal{P}_{o,x}$ is the point process conditioned to have points at both $o$ and $x.$  By invariance of $\mathcal{P},$ the palm process $\mathcal{P}_{o,x}$ is invariant under any isometry of $\mathbb{M}$ fixing the set $\left\{ o,x \right\},$ in particular $\tau_{o,x},$ and hence
  \[
    \begin{aligned}
    \Exp_{\mathcal{P}_{o,x}}\left[ 
      f(G,o,x,\psi)\one\left[ \mathcal{C}_{o,x}(\mathcal{P}_{o,x}) \right]
    \right]
    &=
    \Exp_{\tau_{o,x}^2(\mathcal{P}_{o,x})}\left[ 
      f(G,o,x,\psi)\one\left[ \mathcal{C}_{o,x}(\mathcal{P}_{o,x}) \right]
    \right] \\
    &=
    \Exp_{\tau_{o,x}(\mathcal{P}_{o,x})}\left[ 
      f(G,x,o,\tau_{o,x}\circ \psi)\one\left[ \mathcal{C}_{o,x}(\tau_{o,x}(\mathcal{P}_{o,x})) \right]
    \right] \\
    &=
    \Exp_{\mathcal{P}_{o,x}}\left[ 
      f(G,x,o,\psi)\one\left[ \mathcal{C}_{o,x}(\mathcal{P}_{o,x}) \right]
    \right].
  \end{aligned}
  \]
  In the second equation, we have changed the integration to be over the reflected point process.  The graph topology does not change on doing this, however the embeddings change, as they correspond to the locations of the points.  In the final line, we use the distributional invariance of $\mathcal{P}_{o,x}.$  On changing the integration to be against this random variable, we no longer need to reflect the embeddings $\psi.$ Integrating this against $\Lambda_{o}(dx),$ we conclude that
  \[
    \Exp\left[ 
      \sum_{x \in \mathcal{P}_o} f(G,o,x,\psi)
    \right]
    =
    \Exp\left[ 
      \sum_{x \in \mathcal{P}_o} f(G,x,o,\psi)
    \right].
  \]
\end{proof}

We turn to the proof of Theorem~\ref{thm:la}.
This will depend greatly on the unimodularity, which we will use frequently in the following way.

\begin{lemma}
  Suppose $f$ is a Borel map from embedded rooted Delaunay graphs considered up to isomorphism to the nonnegative real numbers.  Let $(G,o,\psi)$ be the embedded Delaunay graph of an invariant Poisson point process $\mathcal{P}_o.$  Let $(G,o,\psi \times \omega)$ be a finitary unimodular percolation, and let $K_\omega(v)$ be the connected component in this percolation of $v.$
  \[
    \Exp\left[ f(G,o,\psi\times \omega) \right]
    =
    \Exp\left[ \sum_{v \in K_\omega(o)}\frac{f(G,v,\psi \times \omega)}{|K_\omega(o)|} \right].
  \]
  \label{lem:unihelper}
\end{lemma}
\begin{proof}
  We define a new nonnegative Borel function $g$ of birooted networks by the rule
  \[
    g(G,x,y,\psi\times\omega) = \frac{f(G,x,\psi\times\omega)\one\left[ K_\omega(x) = K_\omega(y) \right]}{|K_\omega(x)|}.
  \]
  Then, on the one hand
  \[
    \Exp\left[ \sum_{y \in \mathcal{P}_o} g(G,o,y,\psi\times\omega)\right] =
    %\Exp\left[ \sum_{y \in \mathcal{P}_o}\frac{f(G,x,\psi\times\omega)\one\left[ K_\omega(x) = K_\omega(y) \right]}{|K_\omega(x)|}\right] =
    \Exp\left[ f(G,o,\psi\times \omega)  \right].
  \]
  On the other hand,
  \[
    \Exp\left[ \sum_{y \in \mathcal{P}_o} g(G,y,o,\psi\times\omega)\right]
    =
    \Exp\left[ \sum_{v \in K_\omega(o)}\frac{f(G,v,\psi \times \omega)}{|K_\omega(o)|} \right],
  \]
  and the equality of these two follows from unimodularity.
\end{proof}

We also need the following simple geometric observation about Voronoi tessellations.
\begin{lemma}
  Let $S \subset \mathbb{M}$ be a closed discrete set, and let $r>0$ be arbitrary.  Suppose $o,x \in S,$ have that $\MB(o,r) \cap \vso(x; S) \neq \emptyset.$  Then $\dM(x,o) \leq 2r.$
  \label{lem:vball}
\end{lemma}
\begin{proof}
  Suppose there is a point $y \in \MB(o,r) \cap \vso(x; S).$ Then $\dM(o,y) \leq r,$ and $\dM(x,y) \leq \dM(y,o)$ as $y \in \vso(x;S).$  Hence $\dM(o,x) \leq 2r$ by the triangle inequality.
\end{proof}

With these observations in place, we turn to proving Theorem~\ref{thm:la}.
\begin{proof}

  The implication that $(2) \implies (3)$ is trivial.  We show that $(1) \implies (2)$ and that $(3) \implies (1),$ beginning with the former.  Let $\mathcal{P}_o$ be a distributional lattice, and let $(G,o,\psi)$ be its embedded Delaunay graph.  We wish to show that for any $\epsilon>0,$ there is a finitary unimodular percolation $\omega$ for which
  \(
  (G,o,\psi\times \omega)
  \)
has
\[
  \Exp\left[ \frac{|\partial_E K_\omega(o)|}{|K_\omega(o)} \right] \leq \epsilon,
\]
with $K_\omega(o)$ the connected component of $o.$
By Lemma~\ref{lem:unihelper},
\begin{equation}
  \label{eq:lb0}
  \Exp\left[ \frac{|\partial_E K_\omega(o)|}{|K_\omega(o)} \right] 
  =
  \Exp\left[ 
    \sum_{x : d_G(o,x)=1} \one[ \omega(\{x,o\})=1]
  \right].
\end{equation}
Let $\delta >0$ and let $\mathcal{Q}$ be a stationary Poisson process on $\mathbb{M}$ with intensity $\delta \cdot \VolM(\cdot).$  Define a percolation on $(G,o,\psi)$ by letting $\omega(v) = 1$ for all $v \in \mathcal{P}_o$ and letting 
\[
  \omega( \{x,y\} ) = \one[ \exists~q \in \mathcal{Q}~:~x,y \in \vso(q ; \mathcal{Q})],
\]
that is to say that $x$ and $y$ are connected if and only if they are in the same Voronoi cell formed by $\mathcal{Q}.$  From the almost sure finiteness of the Voronoi cells of $\mathcal{Q},$ this percolation is therefore finite.

Let $d_1$ be the distance of the closest point in $\mathcal{Q}$ to $o,$ and let $d_2$ be the distance of the second closest point in $\mathcal{Q}.$
Then
\[
  \Pr\left[ d_1 \geq t \right] = \exp(-\delta \VolM( \MB(o,t))),
\]
and conditional on $d_1,$
\[
  \Pr\left[ d_2-d_1 \geq t ~|d_1 \right] = \exp\biggl(-\delta \biggl(\VolM( \MB(o,d_1+t))-(\VolM( \MB(o,d_1)))\biggr)\biggr).
\]
Hence, integrating out the dependence on $d_1,$ we get that
\[
\Pr\left[ d_2-d_1 \geq t\right]
=
\int_0^\infty
\exp\biggl(-\delta \biggl(\VolM( \MB(o,u+t))\biggr)
\cdot \delta \partial_u(\VolM( \MB(o,u)))\,du.
\]
Observe that 
\[
\VolM( \MB(o,u+t))
\geq
\VolM( \MB(o,u))
+\VolM( \MB(o,t/2)),
\]
as it is possible to fit disjoint balls of radius $u$ and $t/2$ inside one of radius $u+t,$ and hence
\[
\Pr\left[ d_2-d_1 \geq t\right]
\leq
e^{-\delta \VolM( \MB(o,t/2))}
\cdot
\int_0^\infty
-\partial_u
\exp\biggl(-\delta \biggl(\VolM( \MB(o,u))\biggr)\,du.
\]
In particular, this probability goes to $1$ for any fixed $t > 0$ as $\delta \to 0.$

The relevance of this calculation is that if $o \in \vso(q; \mathcal{Q}),$ then the distance $\dM(o,q) = d_1.$  If $x \in \mathbb{M}$ is any point such that $\dM(o,x) < (d_2-d_1)/2,$ then 
\[
  \dM(x,q) < d_2/2 + d_1/2.
\]
On the other hand, if $q' \in \mathcal{Q}$ is any other point
\[
  \dM(x,q') \geq \dM(o,q') - (d_2-d_1)/2 \geq d_2/2 + d_1/2.
\]
Hence, we have shown that $\MB(o,(d_2-d_1)/2) \subset \vso(q; \mathcal{Q}).$
Therefore, we have that
\[
 \Exp\left[ 
    \sum_{x : d_G(o,x)=1} \one[ \omega(\{x,o\})=0]
  \right] 
  \leq 
 \Exp\left[ 
   \sum_{x : d_G(o,x)=1} \one[ \dM(x,o) \geq (d_2-d_1)/2]
  \right]. 
\]
Since $(d_2-d_1)/2 \to \infty$ in probability as $\delta \to 0,$ we get by dominated convergence and \eqref{eq:lb0}
\[
  \lim_{\delta \to 0}
  \Exp\left[ 
    \sum_{x : d_G(o,x)=1} \one[ \dM(x,o) \geq (d_2-d_1)/2]
  \right] = 0,
\]
and so it can be chosen smaller than $\epsilon,$ which completes the proof by \eqref{eq:lb0}.

  We turn to showing the implication $(3) \implies (1).$  Let $(G,o,\psi)$ be an invariantly amenable embedded Delaunay graph of a distributional lattice $\mathcal{P}_o.$ 
Let $S$ be a compact subset of the isometry group of $\mathbb{M}.$  
By Lemma~\ref{lem:vball}, for any $r>0$ 
\[
  \Pr\left[ \MB(o,r) \not\subset \vso(o) \right] 
  \leq 
  \Pr\left[ \MB(o,2r) \cap \mathcal{P}_o \neq \left\{ o \right\} \right]
  \to 0,
\]
as $r \to 0$ as $\mathcal{P}_o$ is a simple point process. Hence for some $\delta >0,$
\[
  \Pr\left[ \VolM(\vso(o)) \leq \delta \right] \leq \frac{1}{100}.
\]
Let $\epsilon > 0$ be arbitrary.  In what follows, we let $(G,o,\psi \times \omega)$ be any finitary percolation, and let $K_\omega(o)$ be its component containing $o.$
We will use this percolation to construct a finite volume measurable $V \subset \mathbb{M}$ which satisfies the condition
\begin{equation}
  \label{eq:isop}
  \max_{g \in S} \frac{\VolM(gV \triangle V)}{\VolM(V)} \leq C(\delta)\epsilon
\end{equation}
for some $C(\delta) > 0.$  Hence, on taking $\epsilon \to 0$ along some sequence, this gives the desired conclusion.  The set in question will be given by
\[
  \mathscr{V} = \bigcup_{v \in K_\omega(o)} \vso(v),
\]
and we will show there is positive probability that this set satisfies the desired condition \eqref{eq:isop} for some constant $C(\delta).$  As this occurs with positive probability, the desired set must exist.

We begin by observing that the average volume of a Voronoi cell in $K_\omega(o)$ can not be too small.  Using Lemma~\ref{lem:unihelper},
\[
  \Exp\left[ 
    \sum_{v \in K_\omega(o)}
    \frac{
      \one\left[ \VolM(\vso(v)) \leq \delta \right]
    }{|K_\omega(o)|}
  \right]
  =
  \Pr\left[ \VolM(\vso(o)) \leq \delta \right] \leq \frac{1}{100}.
\]
Hence on applying Markov's inequality, we conlude
\begin{equation}
  \Pr\left[ 
     \sum_{v \in K_\omega(o)}
    \frac{
      \one\left[ \VolM(\vso(v)) \leq \delta \right]
    }{|K_\omega(o)|}
    \geq \frac{1}{2} 
  \right]
  \leq \frac{1}{50}.
  \label{eq:la1}
\end{equation}
Hence with probability at least $49/50$ we have that $\VolM(\mathscr{V}) \geq \frac{1}{2}\delta|K_\omega(o)|.$

Conversely, the contribution of large Voronoi cells can not be too big. From dominated convergence,
\[
  \lim_{M \to \infty}\Exp\left[ \VolM(\vso(o))\one[\VolM(\vso(o)) \geq M] \right] \to 0.
\]
Applying Lemma \ref{lem:unihelper}, we can find an $M$ sufficiently large that
\[
  \Exp\left[ 
    \sum_{v \in K_\omega(o)}
    \frac{
      \VolM(\vso(v))\one[\VolM(\vso(v)) \geq M]
    }{|K_\omega(o)|}
  \right]
  \leq \epsilon.
\]
Applying Markov's inequality,
\begin{equation}
  \Pr\left[ 
     \sum_{v \in K_\omega(o)}
    \frac{
      \VolM(\vso(v))\one[\VolM(\vso(v)) \geq M]
    }{|K_\omega(o)|}
    \geq 50\epsilon 
  \right]
  \leq \frac{1}{50}.
  \label{eq:la1a}
\end{equation}

We turn to arguing that while Voronoi cells may be large (in fact they could be unbounded for a nonuniform lattice in a negatively curved space such as $\text{SL}_2(\Z) < \text{SL}_2(\R)$), their volume is necessarily mostly contained in a bounded region.
Specifically just by dominated convergence and the assumption that $\Exp \VolM(\vso(o)) < \infty$ almost surely,
\[
  \lim_{R \to \infty} \Exp[\VolM(\vso(o) \setminus \MB(o, R))] = 0.
\]
Hence, by Lemma~\ref{lem:unihelper} there is an $R$ sufficiently large that
\[
  \Exp\left[ 
    \sum_{v \in K_\omega(o)}
    \frac{
    \VolM(\vso(v) \setminus \MB(v, R))
    }{|K_\omega(v)|}
  \right] < \epsilon.
\]
Applying Markov's inequality, 
\begin{equation}
  \Pr\left[ 
     \sum_{v \in K_\omega(o)}
    \frac{
    \VolM(\vso(v) \setminus \MB(v, R))
    }{|K_\omega(v)|}
    \geq 50\epsilon 
  \right]
  \leq \frac{1}{50}.
  \label{eq:la2}
\end{equation}
By compactness of $S,$ there is an $N>0$ so that $\dM(x, s(x)) \leq N$ for all $s \in S$ and all $x \in \MB(o,R).$

We next show that most Voronoi cells have the property that shifting them by some $s \in S$ leaves them in the interior of $\mathscr{V}.$  From the finiteness of the intensity of $\mathcal{P}_o,$ we have that
\[
  \lim_{T \to \infty}\Pr\left[ |\mathcal{P}_o \cap \MB(o,2(R+N))| \geq T \right] \to 0.
\]
Hence we may find a $T$ sufficiently large that this probability is strictly less than $\epsilon.$
By Lemma \ref{lem:vball}, this $T$ also bounds the number of Voronoi cells intersecting $\MB(o,R+N).$  Using Lemma~\ref{lem:unihelper} and Markov's inequality in the same way as in \eqref{eq:la1} and \eqref{eq:la2},
\begin{equation}
  \Pr\left[ 
     \sum_{v \in K_\omega(o)}
    \frac{
      \one[|\mathcal{P}_o \cap \MB(v,2(R+N))| \geq T]
    }{|K_\omega(v)|}
    \geq 50\epsilon 
  \right]
  \leq \frac{1}{50}.
  \label{eq:la3}
\end{equation}

By assumption there is a percolation $\omega$ so that
\begin{equation*}
%  \label{eq:isop}
  \Exp\left[ \frac{|\partial_E K_\omega(v)|}{|K_\omega(v)|} \right] < \frac{\epsilon}{1+T}.
\end{equation*}
And hence, 
\begin{equation}
  \Pr\left[ 
    \frac{|\partial_E K_\omega(v)|}{|K_\omega(v)|} 
    \geq \frac{50\epsilon}{1+T}
  \right]
  \leq \frac{1}{50}.
  \label{eq:la4}
\end{equation}

From here on, we will work under the event that the complements of the union of events whose probability is estimated in \eqref{eq:la1},\eqref{eq:la1a},\eqref{eq:la2},\eqref{eq:la3},and \eqref{eq:la4}.  This complementary event has probability at least $45/50.$

Call a vertex $v \in K_\omega(o)$ \emph{good} if every nucleus of every Voronoi cell intersecting $\MB(v,R+N)$ is not {adjacent} to a vertex of $K_\omega(o)^c,$ and call it bad otherwise.  Let $B$ denote the set of bad vertices, and let $Q$ be the set of bad vertices $v \in K_\omega(o)$ that have
\[
  |\mathcal{P}_o  \cap \MB(v, 2(R+N))| \leq T,
\]
then
\(
  |Q| \leq T \cdot |\partial_E{K_\omega(o)}| \leq \epsilon \cdot |K_\omega(o)|.
\)
as $|Q|$ is at most a $T$--fold overcounting of the number of vertices in $K_\omega(o)$ adjacent to something outside $K_\omega(o),$ which is itself a lower bound for $|\partial_E{K_\omega(o)}|.$  Hence, the total number of bad vertices is at most $51 \epsilon |K_\omega(o)|.$

Now for any good vertex $v$ and any $s \in S$ we have 
\[
  s(\vso(v) \cap \MB(v,R)) \subset \MB(v,R+N) \subset \mathscr{V}.
\]
Therefore, we have that
\begin{align*}
  \VolM(s \mathscr{V} \setminus \mathscr{V})
  &\leq \sum_{v \in K_\omega(o)} \VolM(\vso(v) \setminus \MB(v,R))
  + \sum_{v \in B} \VolM(\vso(v)) \\
  &\leq 50\epsilon |K_\omega(o)| + M|B| + \sum_{v \in K_\omega(o)} \VolM(\vso(v)) \geq M\one[\VolM(\vso(v))] \\
  &\leq 50\epsilon |K_\omega(o)| + 51 M \epsilon |K_\omega(o)| + 50\epsilon |K_\omega(o)|.
\end{align*}
Now, as $|K_\omega(o)| \leq 2\delta^{-1} \VolM(\mathscr{V}),$ we have shown that
\[
  \VolM(s \mathscr{V} \setminus \mathscr{V})
  \leq (200+102M)\epsilon \delta^{-1} \VolM(\mathscr{V}).
\]
The same bound holds for 
\[
  \VolM(\mathscr{V} \setminus s(\mathscr{V}))
  =\VolM(s^{-1}(\mathscr{V}) \setminus \mathscr{V}),
\]
on account of $s^{-1}$ satisfying the same displacement bound as $s,$ and so we have shown
\[
  \VolM(s \mathscr{V} \triangle \mathscr{V})
  \leq 2(200+102M)\epsilon \delta^{-1} \VolM(\mathscr{V}) := \epsilon C(\delta) \VolM(\mathscr{V}), 
\]
as desired.

\end{proof}

\section{General Properties of Poisson--Voronoi tilings in Symmetric Spaces}
\label{sec:pv}
In this section, we study in some more detail the Poisson--Voronoi tessellations.  We will let $\Pi^\lambda$ denote a Poisson point process with intensity $\lambda \cdot \VolM.$
The main purpose of this section is to show that these tessellations are always distributional lattices, i.e.\,they have finite expected volume and degree.  In fact, more is true.

We will let $f(r) = \VolM(\MB(o,r)),$ the volume growth function.  We will need some simple estimates on this function, which we summarize as follows.
\begin{lemma}
  Suppose $\mathbb{M}$ is a noncompact Riemannian symmetric space.  Then there are constants $c_1$ and $c_2 >0$ so that for all $r \geq 1,$
  \[
    e^{c_1 r} \geq f(r) \geq c_2 r.
  \]
  Further $\lim_{r \to \infty} f(r)^{1/r}$ exists.
  \label{lem:vest}
\end{lemma}
\begin{proof}
  By homogeneity, the scalar curvature is bounded below by some $-\alpha < 0.$  From this lower bound, we get an exponential upper bound on the volume growth by comparing with the corresponding constant curvature space of the corresponding dimension (see for example \cite[Theorem 11.1/2]{Lee}).

  When $\mathbb{M}$ is noncompact, there is an infinite geodesic $\gamma$ connecting $o \to \infty,$ and hence we can bound the volume growth below by comparing with the portion of the volume of a unit--distance tubular neighborhood of $\gamma$ that is contained within $\MB(o,r),$ which gives rise to the lower bound.

  The existence of the limit follows from a subadditivity argument. 
\end{proof}

We use this to find a tail bound for the diameter of $\vso(o; \Pi^\lambda).$
\begin{lemma}
  Suppose $\mathbb{M}$ is a noncompact Riemannian symmetric space. There is a constant $C>0$ so that for all $R>1,$
  \[
    \Pr\left[ \vso(o;\Pi^\lambda) \not \subset \MB(o,R) \right]
    \leq Cf(R) \exp(-\lambda f(R-1)).
  \]
  \label{lem:vtail}
\end{lemma}
\begin{proof}
  Let $R > 0$ be fixed.  Let $\{\Delta_j\}_1^M$ be a minimal $1$--net of the sphere $\partial \MB(o,R).$  By a sphere packing argument, $M \leq C \cdot f(R)$ for some $C>0$ depending only on $\mathbb{M}.$  Hence,
  \[
    \Pr\left[ \exists~1 \leq j \leq M:\MB(\Delta_j, R-1) \cap \Pi^\lambda =\emptyset \right] \leq Cf(R) \exp(-\lambda f(R-1)).
  \]
  If $q \in \Pi^\lambda$ is a point whose cell is adjacent to $\vso(o;\Pi^\lambda),$ then there is a point $u$ with $r=\dM(u,q)=\dM(u,o)$ and so that $\MB(u,r) \cap \Pi^\lambda =\emptyset.$  If $r \geq R,$ then if we let $u'$ be the point on the geodesic from $u$ to $o$ with $\dM(u',o)=R,$ the ball $\MB(u',R) \subset \MB(u,r).$  Further, there is a closest point $\Delta_j$ to $u'$ which is necessarily at distance less than $1$ from it.  Hence $\MB(\Delta_j, R-1) \subset \MB(u',R),$ and so is again empty.  Thus,
  \[
    \Pr\left[ \vso(o;\Pi^\lambda) \not \subset \MB(o,R) \right]
    \leq
    \Pr\left[ \exists~1 \leq j \leq M : \MB(\Delta_j, R-1) \cap \Pi^\lambda =\emptyset \right],
  \]
  completing the proof.
\end{proof}

\begin{theorem}
  \label{thm:pv}
  Let $\mathbb{M}$ be a Riemannian symmetric space.
  Let $\lambda >0$ be arbitrary, and let $\mathcal{P}^\lambda$ be a stationary Poisson process on $\mathbb{M}.$
  Then for any $k \geq 0,$ 
  \[
    \Exp\left[ \VolM(\vso(o;\Pi^\lambda))^k \right] < \infty
    \quad\quad
    \text{and}
    \quad\quad
    \Exp\left[ \deg(o)^k \right] < \infty.
  \]
\end{theorem}
\begin{proof}
  In the case that $\mathbb{M}$ is compact, the result is trivial, as the volume of the whole manifold is finite and the degree is bounded by the total number of points, which is $\Poisson(\lambda \VolM(\mathbb{M})).$  

  In the case that $\mathbb{M}$ is noncompact, we use the bound in Lemma~\ref{lem:vtail}.  Conditional on the smallest $R$ so that
  \[
    \vso(o;\Pi^\lambda) \subset \MB(o,R),
  \]
  we can estimate $\VolM(\vso(o;\Pi^\lambda)) \leq \VolM(\MB(o,R)) = f(R).$
  Hence,
  \begin{equation}
    \label{eq:vt1}
    \Exp[\VolM(\vso(o;\Pi^\lambda))^k]
    \leq \int_0^\infty k f'(R) f(R)^k C\exp(-\lambda f(R-1))\,dR.
  \end{equation}
  
  Since $\lim_{R\to\infty} f(R)^{1/R}$ exists, for any $\delta>0$ we can estimate $f(R)^k \leq C_{\delta,k} e^{\delta f(R-1)}.$  Likewise, we can estimate $f(R-1) \geq cf(R)$ for some $c \in (0,1)$ for all $R>2.$  In particular, we can estimate the tail of the integral by
  \begin{align*}
  &\int_2^\infty f'(R) f(R)^k \exp(-\lambda f(R-1))\,dR \\
  \leq C &\int_2^\infty f'(R) \exp(-\eta f(R))\,dR 
  < \infty,
\end{align*}
for some constant $\eta >0.$

  Turning to the degree bound, with the same random $R$ as used to control the volume, we have that all neighbors of $o$ are contained in $\MB(o,2R).$  Hence
  \begin{align*}
    \Exp[\deg(o)^k] 
    &\leq \sum_{j=0}^\infty \Exp\left[ |\Pi^\lambda \cap \MB(o,2j+2)|^{k} \one[j+1 \geq R \geq j] \right] \\
    &\leq \sum_{j=0}^\infty \Exp\left[ |\Pi^\lambda \cap \MB(o,2j+2)|^{k} \one[ R \geq j] \right] \\
    &\leq \sum_{j=0}^\infty \biggl(\Exp\left[ |\Pi^\lambda \cap \MB(o,2j+2)|^{2k}\right] \Pr[ R \geq j] \biggr)^{1/2} \\
    &\leq C_0 + C(k,\lambda)\sum_{j=2}^\infty f(2j+2)^{k} f(j)^{1/2}e^{-\lambda f(j-1)/2} \\
    &\leq C_0 + C(k,\lambda)'\sum_{j=2}^\infty f(j-1)^{2k+1}e^{-\lambda f(j-1)/2} < \infty,
  \end{align*}
  where we have used that the volume grows at most exponentially to compare $f(2j+2) \leq Cf(j-1)^{2}$ for some $C>0$ and all $j \geq 2.$
\end{proof}

\section{Additional structure for Poisson--Voronoi tessellations in nonpositively curved spaces} 
\label{sec:npc}

Non-positive curvature is beneficial for many reasons, one of which is that the notion of convexity translates well to nonpositively curved space (see \cite[Section 1.6]{Eberlein}). We will also deal exclusively with the simply connected case, for which we recall that $\mathbb{M} = \R^d \times \prod_{i=1}^r \mathbb{M}_i$ for some Riemannian symmetric spaces of noncompact type $\mathbb{M}_i.$  In particular, these spaces are amenable if and only if the space is some Euclidean space.  As a consequence, when there is a nonEuclidean factor, the volume growth is necessarily exponential, i.e.\,
\begin{equation}
 \lim_{r \to\infty} f(r)^{1/r} = h > 0.
 \label{eq:volume}
\end{equation}
For Riemannian symmetric spaces of noncompact type, more precise estimates are available \cite{Knieper}.

%Of particular importance, metric balls are geodesically convex. Further, so long as $\mathbb{M}$ is not flat, which is to say that $\mathbb{M}$ is not $\R^d$, $\mathbb{M}$ will have exponential volume growth.  In fact, recalling
%\(
%f(r)=\VolM(\MB(x_0,r)),
%\)
%we have that there are positive constants $k,h,C$ so that for all $r \geq 1,$ 
%\begin{equation}
%  \frac{1}{C} < f(r) r^{-k} e^{-hr} < C
%\label{eq:volume}
%\end{equation}
%(see \cite{Knieper}).  To reiterate, $h=0$ if and only if $\mathbb{M}=\R^k.$  

By a theorem of Borel~\cite{Borel}, in any simply connected Riemannian symmetric space, there is a co-compact lattice $\Gamma.$
This means there is a countable collection of points $\Lambda$ and a group of isometries of $\mathbb{M}$ that act transitively on $\Lambda,$ with the further property that the Voronoi cells with nuclei $\Lambda$ are bounded.  Let $\mathscr{L}$ denote the dual graph of these cells, which by virtue of the transitive action on $\Lambda$ becomes a transitive graph.  To a point $y \in \Lambda,$ we define $\mathcal{L}(y),$ the Voronoi cell in the $\Lambda$-nucleated tessellation that is centered at $y.$  The large-scale geometry of $\mathbb{M}$ is captured by the large-scale geometry of $\mathscr{L}:$ any map $\pi : \mathbb{M} \to \Lambda$ with the property that if $\pi(x) = y$ then $x \in \mathcal{L}(y)$ is a quasi-isometry of the two spaces. 

  \begin{proposition}
	  Let $\mathbb{M}$ be a nonpositively curved, nonamenable, simply connected Riemannian symmetric space.
	  Let $(G,o,\psi)$ be the embedded Delaunay graph with nuclei $\Pi^\lambda$ for some $\lambda >0.$
	  There are constants $\delta > 0$ and $t_0>0$ depending only on $\mathbb{M}$ and $\lambda$ so that for all $t > t_0$ and all $R > 1,$
	  \[
		  \Pr
		  \left[ 
			  B_{G}(o,R)
			  \centernot{\subseteq}
			  \MB(x_0,tR)
		  \right]
		  \leq e^{-Re^{\delta t}}.
	  \]
%	  Here $\phi$ is the function from Lemma~\ref{lem:YS}.
	  As a consequence, there are constants $\alpha,\beta,R_0 > 0$ depending only on $\mathbb{M}$ and $\lambda$ so that for all $R>R_0,$
  \[
		  \Pr
		  \left[ 
			  \left|
			  B_{G}(o,R)
			  \right|
			  \geq e^{\alpha R}
		  \right]
		  \leq e^{-\beta R}.
  \]
  In particular, we have that 
  \[
    \limsup_{r\to \infty}|B_{G}(o,R)|^{1/R} < \infty,
  \]
  almost surely.
  \label{prop:expgrowth}
\end{proposition}

To prove this, we begin with the following Lemma on the volume of a union of balls in a nonamenable space.
\begin{lemma}
	Let $\mathbb{M}$ be a nonamenable Riemannian symmetric space.  There is an $\alpha > 0$ so that the following holds. 
%	Let $k$ be the constant from \eqref{eq:volume}.
	For any $\Delta>0,$ there is a constant $C >0$ so that for any finite collection of balls 
	\(
	\left\{\MB(x_i,r_i)\right\}_{i=1}^t
	\)
	with $\dM(x_i,x_j) \geq \Delta$ whenever $i \neq j,$ we have
	\[
		\VolM\left( \cup_{i=1}^t \MB(x_i,r_i) \right)
		\geq
		%-tC+
		\sum_{i=1}^t 
		\frac{e^{\alpha r_i}}{C}\one[r_i \geq \tfrac \Delta 2].
	\]
	\label{lem:YS}
\end{lemma}
\begin{proof}
	%As $\mathbb{M}$ is a nonpositively curved, nonflat Riemannian symmetric space, it is nonamenable.  
	As $\mathbb{M}$ is nonamenable, there is a constant $h=h(\mathbb{M}) > 0$ so that for all piecewise smooth compact submanifolds $A$ with boundary,
\[
	|\partial A| \geq h \VolM(A),
\]
where $|\partial A|$ denotes the surface measure of $\partial A.$  

Set
\[
	W = \bigcup_{\substack{i\in\left\{1,\dots,t\right\} \\ r_i \geq \Delta/2}} \MB(x_i,r_i).
\]
We have that for any continuous $\phi : \R \to \R^{+},$
setting
$\Phi(x) = \int_0^x \phi(y)\,dy,$
that
\begin{align}
	\int_W \phi( \dM(x,W^c))\,dV(x)
	\geq 
	h \int_W \Phi( \dM(x,W^c))\,dV(x),
	\label{eq:ibp}
\end{align}
by decomposing the integral into level sets of $\dM(x,W^c)$ and applying integration by parts.

Let $\alpha > 0$ be any constant strictly less than $h.$  
Using \eqref{eq:ibp} with $\phi(x) = e^{\alpha x}$, we therefore get that
\[
	\int_W e^{\alpha \dM(x,W^c)}\,dV(x)
\geq
\frac{h}{\alpha} 
\left(
\int_W e^{\alpha \dM(x,W^c)}\,dV(x)
-\VolM(W)
\right).
\]
Rearranging, we get that
\begin{equation}
	\VolM(W)
	\geq
	\frac{h-\alpha}{h}
	\int_W e^{\alpha \dM(x,W^c)}\,dV(x).
	\label{eq:volest}
\end{equation}

We note that there is a constant $C=C(\alpha, \Delta) >0$ so that for any $1 \leq i \leq t$ with $r_i \geq \tfrac \Delta 2,$
\[
e^{\alpha r_i}
\leq C \int_{\MB(x_i,\Delta/2)} e^{\alpha \dM(x,W^c)}\,dV(x).
\]

Hence
\begin{align*}
		\sum_{i=1}^t 
		{e^{\alpha r_i}}\one[r_i \geq \tfrac \Delta 2]
		&\leq
		C
		\sum_{i=1}^t 
		\one[r_i \geq \tfrac \Delta 2]
		\int_{\MB(x_i,\Delta/2)} e^{\alpha \dM(x,W^c)}\,dV(x). \\
		&\leq
		C
		\int_{W} e^{\alpha \dM(x,W^c)}\,dV(x). 
\end{align*}
By \eqref{eq:volest}, the proof is complete.
\end{proof}

\begin{proof}[Proof of Proposition~\ref{prop:expgrowth}]

	Recall that $\Lambda$ is a collection of points on whose Voronoi cells $\{\mathcal{L}(y)\}_{y \in \Lambda}$ on which a subgroup of $\mathbb{M}$ acts transitively.  In particular this implies that all
\(
\{\mathcal{L}(y)\}_{y \in \Lambda}
\)
have the same diameter and volume.  Let $\Delta$ denote this diameter.  

We will need to compare balls in $\mathbb{M}$ to balls in $\mathscr{L},$ the dual graph of the Voronoi tessellation $\{\mathcal{L}(y)\}_{y \in \Lambda}.$
%Note that there is a constant $C$ so that for any ball $\MB(x,r)$ 
%with $r > \Delta.$
Let $x \in \mathbb{M}$ be arbitrary, and
let $y \in \Lambda$ be a point so that $x \in \mathcal{L}(y).$  There are positive constants $r_0,C_1,$ and $C_2$ so that when $r>r_0,$
\begin{equation}
	\label{eq:lattice_balls}
\MB(x,r/C_1)
\subseteq
\sum_{v \in \LB(y,r/C_2)}
{\mathcal{L}(v)} 
\subseteq \MB(x,r)
\end{equation}
%of radius strictly larger than $\Delta$, we therefore have that there is an absolute constant $C>0$ so that
%\begin{equation}
%	\VolM(V)/C \leq \left| \left\{ y \in \Lambda:
%		\mathcal{L}(y) \subseteq V
%	\right\}\right|
%		\leq C\VolM(V),
%\label{eq:lattice_volume}
%\end{equation}
%both of which follow by considering volumes.
%
We let $G$ denote the dual graph of the Poisson Voronoi tessellation.  Let $G'$ be the graph formed by adding to $G$ an edge between every pair of Poisson Voronoi cells that are at distance at most $\Delta.$  Note that this implies that 
\(
B_{G}(o,R)
\subseteq
B_{G'}(o,R)
\)
for all $R>0,$ and hence it suffices to show the claim for $G'.$

Suppose that $x_0x_1x_2\cdots x_R$ is a collection of $\PPP$ whose Poisson Voronoi cells form a geodesic path in $G'.$  Let $J\subseteq \left\{ 0,1,2,\ldots, R-1 \right\}$ be those indices for which $x_i$ and $x_{i+1}$ are adjacent in $G$.  For every $i \in J,$ there is an open ball 
\(
V_i
= \MB(z_i,r_i)
\) 
with $V_i \cap \PPP = \emptyset$ and $\left\{ x_i,x_{i+1} \right\} \subset \partial V_i.$  For $i \centernot{\in} J,$ there are points $u,v$ with $u$ on the boundary of $\PV(x_i)$ and $v$ on the boundary of $\PV(x_{i+1})$ so that $\dM(u,v) \leq \Delta.$  Let $r = \dM(u,x_i)$ and $s = \dM\left( v,x_{i+1} \right).$  Both of $\MB(u, r)$ and $\MB(v,s)$ do not intersect $\PPP.$  Let $z_i$ be the midpoint of $uv.$  Set $r_i = \max(r,s)-\Delta/2.$  Then provided $r_i \geq 0,$ we have $\MB(z_i,r_i)$ is contained in one of $\MB(u, r)$ or $\MB(v,s).$  In this case, let $V_i = \MB(z_i,r_i)$ be this ball, and let $V_i =\emptyset$ otherwise. 

Note that for $i \centernot{\in} J,$ we have that
\[
	\dM(x_{i},x_{i+1})
	\leq
	\dM(x_i,u)
	+\dM(u,v)
	+\dM(v,x_{i+1})
	\leq 2r_i + 2\Delta.
\]
For $i \in J$ we have
by the triangle inequality, $2r_i \geq \dM(x_i,x_{i+1}).$
\begin{align*}
	\dM(x_0,x_R) 
	&\leq
	\sum_{i=0}^{R-1}
	\dM(x_i,x_{i+1}) \\
	&\leq
	\sum_{i \centernot{\in} J} 
	\dM(x_i,x_{i+1}) 
	+\sum_{i \in J} 
	\dM(x_i,x_{i+1}) \\
	&\leq
	2\Delta \cdot R
	+
	\sum_{i=0}^{R-1} 
	2r_i.
\end{align*}
Hence provided that $\dM(x_0,x_R)/R$ is large, so too will be the sum of radii of these balls.

We now show that the centers of these balls are mostly separated from one another.  In fact, by this construction, it could be that the centers of $V_i$ and $V_{i+1}$ are close.  However, it will transpire that if $i+1 < \ell,$ then $V_i$ and $V_{\ell}$ have centers separated by at least $\Delta/2.$  

If $i \in J,$ the center of $V_i,$ which is $z_i,$ is contained in $\PV(x_i) \cap \PV(x_{i+1}).$  For $V_{\ell},$ its center $z_\ell$ is within distance $\Delta/2$ of $\PV(x_{\ell+1}).$  Hence if $\dM(z_i,z_\ell) < \Delta/2,$ the distance between $\PV(x_i)$ and $\PV(x_{\ell+1})$ is at most $\Delta.$  In this case the path
$x_0x_1x_2 \cdots x_i x_{\ell+1} x_{\ell+2} \cdots x_R$ is a shorter path from $x_0$ to $x_R,$ contradicting that the path was a geodesic.  The same proof works if $\ell \in J,$ now using that $z_{\ell}$ is in $\PV(x_{\ell+1}).$

Suppose that both of $i$ and $\ell$ are not in $J.$  Then by construction $z_i$ and $z_\ell$ are at most distance $\Delta/2$ from $\PV(x_i)$ and $\PV(x_{\ell+1})$ respectively.  If $\dM(z_i,z_\ell) < \Delta/2,$ then their midpoint is contained in some $\PV(y).$  The distance between $\PV(y)$ and $\PV(x_i)$ is at most $3\Delta/4,$ as is the distance between $\PV(y)$ and $\PV(x_{\ell+1}).$  Hence the path
$x_0x_1x_2 \cdots x_i y x_{\ell+1} x_{\ell+2} \cdots x_R$ is a shorter path from $x_0$ to $x_R,$ contradicting that the path was a geodesic.

Both of the collections
$\left\{ V_i \right\}_{i~\text{odd}}$
and
$\left\{ V_i \right\}_{i~\text{even}}$
have centers that are $\Delta/2$ separated. Let $I$ be the whichever of the odd integers or the even integers has a larger sum of radii, and let $I_M$ be the subcollection of $I$ for which all $r_i > M.$  Then
\begin{equation}
	\dM(x_0,x_R) 
	\leq 2\Delta \cdot R + 4\sum_{i\in I} r_i
	\leq (2\Delta +M)\cdot R + 4\sum_{i\in I_M} r_i.
	\label{eq:radii_lower}
\end{equation}

Let $y_0 \in \Lambda$ be a point so that $x_0 \in \mathcal{L}(y),$ and let $y_R \in \Lambda$ be a point so that $x_R \in \mathcal{L}(y).$  Also let $y_i \in \Lambda$ be a point so that $z_i \in \mathcal{L}(y).$ By quasi-isometry, we have that for $1 \leq i < R -1,$
\begin{align*}
	\dL(y_i,y_{i+1})
	&\leq C\dM(z_i,z_{i+1}) + C \\
	&\leq C(\dM(z_i,x_i) + \dM(z_{i+1},x_{i+1})  + \dM(x_i,x_{i+1}))+C \\
	&\leq C(r_i + r_{i+1} + \Delta  + 2r_i+2\Delta)+C.
\end{align*}
A similar bound holds for the end points, and we conclude that
\begin{equation}
	\sum_{i=0}^{R-1} 
	\dL(y_i,y_{i+1})
	\leq CR 
	+ 
	C
	\sum_{i=0}^{R-1} 
	r_i
	\leq C(1+M)R + C\sum_{i\in I_M} r_i,
	\label{eq:path_length_bound}
\end{equation}
for some constant $C$ depending on $\Delta.$

By \eqref{eq:lattice_balls},
there is an $M$ and a $D$ sufficiently large so that
for each $i \in I_M$ there is a ball $W_i = \LB(y_i,\rho_i)$ where $\rho_i = r_i/D$ with
\[
\MB(z_i,\rho_i/D) 
\subseteq
\sum_{y \in W_i} \mathcal{L}(y) \subseteq \MB(z_i,r_i).
\]
Hence we have a lattice path $\gamma \in \mathscr{L}$ connecting $y_0$ to $y_R,$ and a collection of lattice balls $\left\{ W_i \right\}_{i \in I_M}$ centered at some points on this path with the property that   
\begin{align*}
	\VolM(\mathcal{L}(0))
	\cdot
	\left|
	\cup_{i \in I_M} W_i
	\right|
	&=
	\VolM( \cup_{i \in I_M} \cup_{y \in W_i} \mathcal{L}(y)) \\
	&\geq
	\VolM( \cup_{i \in I_M} \MB(z_i,r_i/D^2)) \\
	&\geq \sum_{i \in I_M}
	\frac{e^{\alpha r_i/D^2}}{C}
\end{align*}
provided $M > \Delta/4$
by Lemma~\ref{lem:YS}.  
By convexity
\[
	\frac{|I_M|}{|I_M|}
	\sum_{i \in I_M}
	e^{\alpha r_i/D^2}
	\geq
	|I_M|\exp\left( \alpha \frac{1}{|I_M|}\sum_{i \in I_M} r_i/D^2\right).
\]
We conclude that for $M$ sufficiently large, we have
\begin{equation}
	\left|
	\cup_{i \in I_M} W_i
	\right|
	\geq \frac{|I_M|}{C} \exp\left(\alpha \frac{\sum_{i \in I_M} \rho_i}{D|I_M|}\right).
	\label{eq:lattice_volume_sum}
\end{equation}

%To summarize, we have shown that if there is a geodesic of length $R$ in $G'$ that connects from $0$ to some $\dM$ distance $tR$ then by \eqref{eq:radii_lower}, there is a collection of balls $\left\{ V_i \right\}_{i \in I_M}$ whose radii satisfy
%\[
%	\sum_{i \in I_M} r_i \geq (t/C - C)R
%\]
%for some large $C.$  Then there is also a lattice path $\gamma$ whose $\mathscr{L}$-length $\ell$ is at most
%\[
%\ell \leq CR+C\sum_{i \in I_M} r_i,
%\]
%and on this lattice path there are balls with 

Let $\ell$ denote the length of $\gamma,$
and let $k = |I_M|.$
The number of paths of length $\ell$ is $C^\ell,$ as $\mathscr{L}$ is a regular graph.
The number of possible ways to choose which vertices will be centers of $W_i$ is at most $\binom{\ell+k}{k}.$  Let $s= \frac{\sum_{i \in I_M} \rho_i}{Dk}.$  Then since for some constant $C>0,$ 
\(
\ell \leq Ck(1+s),
\)
we have that the probability that $(\gamma, \left\{ W_i \right\}_{i \in I_M})$ could exist for some particular choice of $(\rho_i)$ is
\begin{align*}
	\Pr
	\left[
		\exists (\gamma, \left\{ W_i \right\}_{i \in I_M})
		\text{ for a specific }
		\left( \rho_i \right)_{i=1}^k
	\right]
	&\leq \sum_{\ell = 1}^{Ck(1+s)}
	(C)^{\ell}
	\binom{\ell+k}{k}
	p^{\frac{k}{C} e^{\alpha s}}, \\
	\intertext{where $p$ is the probability that a cell of $\mathscr{L}$ is empty.  Adjusting constants, we have that } 
	\Pr
	\left[
		\exists (\gamma, \left\{ W_i \right\}_{i \in I_M})
		\text{ for a specific }
		\left( \rho_i \right)_{i=1}^k
	\right]
	&\leq e^{ 
	k( C(1+s)
	-e^{\alpha s}/C )}.
\end{align*}

This we can now sum over all choices of integers $(\rho_i)_{i=1}^k$ with $s > s_0 = t_0R/k$ to conclude
\begin{align*}
	\Pr
	\left[
		\exists (\gamma, \left\{ W_i \right\}_{i \in I_M}), |I_M|=k
	\right]
	&\leq
	\sum_
	{
		\substack{\left( \rho_i \right) \in \mathbb{N}_0^k
		\\ 
	s \geq s_0}
}
e^{ 
	k( C(1+s)
	-e^{\alpha s}/C + C)}\\
	&\leq
	\int
	\limits
	_
	{\substack{\left( \rho_i \right) \in \mathbb{R}^R \\ 
	s \geq s_0}}
		Ce^{ 
	k( C(1+s)
	-e^{\alpha s}/C + C)}\,d\rho_1d\rho_2\cdots d\rho_k.\\
	\intertext{
	The simplex $\left\{ \left( \rho_i \right)_{i=1}^k, \rho_i \geq 0, \sum_{i=1}^k \rho_i = Dks \right\}$ has Euclidean volume $O((Cs)^k).$  Hence we
	can change the integration to be over $s,$ so that after adjusting constants, we have}
	\Pr
	\left[
		\exists (\gamma, \left\{ W_i \right\}_{i \in I_M}), |I_M|=k
	\right]
&\leq
	\int
	\limits_{s_0}^\infty
	e^{ 
	k( C(1+s+\log s)
	-e^{\alpha s}/C )}\,ds.
\end{align*}
By \eqref{eq:volume},
provided that $t_0$ is chosen sufficiently large, we then get that this integral converges and is dominated by its value at $s_0.$  Summing over $k$ and adjusting constants, we get
\begin{equation}
	\Pr
	\left[
		\exists (I_M,\gamma, \left\{ W_i \right\}_{i \in I_M}
		)
		\text{ so that }
		{\sum_{i \in I_M} \rho_i} > tD|I_M|
	\right]
	\leq
	e^{ 
	R( C(1+t)
	-e^{\alpha t}/C )}
	\label{eq:finalbound}
\end{equation}
for all $t \geq t_0.$

Hence, with the same probability, by \eqref{eq:radii_lower}, we have
\begin{equation}
	\dM(x_0,x_R) 
	\leq (2\Delta +M)\cdot R + 4\sum_{i\in I_M} \rho_i D
	\leq (2\Delta +M + 4D^2t)\cdot R.
	\nonumber
\end{equation}
This proves the first part of the proposition.

To prove the second part of the proposition, note that by the first part, we can find some $t,\beta,$ and $R_0$ so that for all $R>R_0,$
\[
\Pr
		  \left[ 
			  B_{G}(o,R)
			  \centernot{\subseteq}
			  \MB(x,tR)
		  \right]
		  \leq e^{-\beta R}.
\]
On the event that
\(
 B_{G}(o,R) \subseteq \MB(x,tR),
\)
we have that $|B_{G}(o,R)| \leq | \PPP \cap \MB(x,tR)|.$  And by Poisson tails, we have that
\[
	\Pr [
| \PPP \cap \MB(x,tR)| \geq 2\lambda f(tR)
] \leq e^{-\lambda f(tR)/C}
\]
for some constant.  Hence, combining this with the previous bound and adjusting constants gives the second consequence.  The third part of the proposition follows immediately by Borel-Cantelli.

\end{proof}

%As with Section~\ref{sec:hpv}, we will assume here that $\PPP$ has a point at the origin.
\begin{proposition}
  When $\mathbb{M}$ is a Riemannian symmetric space of noncompact type, the embedded Delaunay network $(G,o,\psi)$ with nuclei $\Pi^\lambda$ is a random weak limit of finite random networks.
  \label{prop:sofic}
\end{proposition}
\begin{remark}
  This was observed earlier in~\cite[``Hyperbolic Surfaces'' proof of Theorem 6.2]{BenjaminiSchramm01}.
\end{remark}
\begin{proof}
  The core of the proof is the existence of the following family of spaces (see \cite[Theorem 2.1]{dGW}).
  There is a family of Riemannian manifolds $\{S_r\}_{r=1}^\infty$ so that $S_r$ has the property that any ball of radius $r$ in $S_r$ is isometric to $\MB(0,r) \subseteq \mathbb{M}.$ Hence, on $S_r$ we can define a Poisson point process $\PPP_{S_r}$ whose intensity measure on any ball of radius $r$ is the pullback of the intensity of $\PPP$ on $\MB(0,r).$  We can also associate to $\PPP_{S_r}$ its associated Voronoi tessellation, and we define $(G_r,\rho_r,\psi_r)$ to be the embedded Delaunay network, where $\rho_r$ is a uniformly chosen vertex of $G_r.$  We claim that $(G,o,\psi)$ is the local limit of $(G_r,\rho_r,\psi_r),$ i.e.\,$(G,o,\psi)$ is the random weak limit of $G_r.$  
  %Let $V_r \subset \mathbb{M}$ be the union of all Voronoi cells with nuclei in $B_{G}(\rho,r).$  
  From the almost sure finiteness of the Voronoi cells with nuclei $\PPP$ on $\mathbb{M},$ we have that for each $r > 0,$
  \[
    \lim_{q \to \infty} 
    \Pr\left[ 
      \max_{x \in B_{G}(\rho,r)}
    \diamM(\PV(x)) > q
    \right] = 0.
  \]
  Hence by diagonalization, we can find some sequence $\left\{ q_r \right\}_{r=1}^\infty \subset \mathbb{N}$ with $q_r \to \infty$ so that 
 \[
    \lim_{r \to \infty} 
    \Pr\left[ 
      \max_{x \in B_{G}(\rho,r)}
    \diamM(\PV(x)) > q_r
    \right] = 0.
  \]
  Let $\mathcal{E}_r$ be the event 
  \(
  \left\{  \max_{x \in B_{G}(\rho,r)}
    \diamM(\PV(x)) \leq q_r\right\}.  
    \)
    On $\mathcal{E}_r,$ every nucleus $x \in B_{G}(\rho,r)$ is contained in $\MB(x_0, q_r \cdot (r+1)).$  Further, if $y \in \PPP$ is a neighbor of some $x \in B_{G}(\rho,r),$ then there is a point $z \in \PV(x) \cap \PV(y)$ that is equidistant to $x$ and to $y.$  Hence $\dM(x,y) \leq 2\dM(x,z) \leq 2q_r.$  In particular, the event $\mathcal{E}_r$ is measurable with respect to $\PPP \cap \MB(x_0,q_r \cdot(r+2)).$

    For every $r \in \mathbb{N},$ let
    \[
	    j(r) = \max \left\{ s \in \mathbb{N} : q_s \cdot (s+2) \leq r \right\}.
    \]
    Then we have that
    \[
    \lim_{r \to \infty} 
    \Pr\left[ 
      \max_{x \in B_{G_r}(\rho_r,j(r))}
      \operatorname{diam}_{S_r}(\PV(x)) > q_{j(r)}
    \right] = 0.
    \]
    Moreover, the law of $ B_{G_r}(\rho_r,j(r))$ on the event
    \(\max_{x \in B_{G_r}(\rho_r,j(r))}
      \operatorname{diam}_{S_r}(\PV(x)) \leq q_{j(r)}
      \)
      coincides with the law of $B_{G}(\rho,r)$ on the event $\mathcal{E}_{j(r)},$ from which the local weak convergence follows. 
      \end{proof}

      \bibliographystyle{amsalpha}
\bibliography{anchored}

\newcommand{\etalchar}[1]{$^{#1}$}
\providecommand{\bysame}{\leavevmode\hbox to3em{\hrulefill}\thinspace}
\providecommand{\MR}{\relax\ifhmode\unskip\space\fi MR }
\providecommand{\MRhref}[2]{%
  \href{http://www.ams.org/mathscinet-getitem?mr=#1}{#2}
}
\providecommand{\href}[2]{#2}
\begin{thebibliography}{AHNR16}

\bibitem[AHNR16]{AHNR}
Omer Angel, Tom Hutchcroft, Asaf Nachmias, and Gourab Ray, \emph{Unimodular
  hyperbolic triangulations: circle packing and random walk}, Invent. Math.
  \textbf{206} (2016), no.~1, 229--268. \MR{3556528}

\bibitem[AL07]{AL}
David Aldous and Russell Lyons, \emph{Processes on unimodular random networks},
  Electron. J. Probab. \textbf{12} (2007), no. 54, 1454--1508. \MR{2354165}

\bibitem[BC12]{BenjaminiCurien}
Itai Benjamini and Nicolas Curien, \emph{Ergodic theory on stationary random
  graphs}, Electron. J. Probab. \textbf{17} (2012), no. 93, 20. \MR{2994841}

\bibitem[BH99]{BridsonHaefliger}
Martin~R. Bridson and Andr{\'e} Haefliger, \emph{Metric spaces of non-positive
  curvature}, Grundlehren der Mathematischen Wissenschaften [Fundamental
  Principles of Mathematical Sciences], vol. 319, Springer-Verlag, Berlin,
  1999. \MR{1744486 (2000k:53038)}

\bibitem[Bor63]{Borel}
Armand Borel, \emph{Compact clifford-klein forms of symmetric spaces}, Topology
  \textbf{2} (1963), no.~1, 111--122.

\bibitem[BPP14]{BPP}
I.~{Benjamini}, E.~{Paquette}, and J.~{Pfeffer}, \emph{{Anchored expansion,
  speed, and the hyperbolic Poisson Voronoi tessellation}}, submitted (2014).

\bibitem[BS01]{BenjaminiSchramm01}
Itai Benjamini and Oded Schramm, \emph{Percolation in the hyperbolic plane},
  Journal of the American Mathematical Society \textbf{14} (2001), no.~2,
  487--507.

\bibitem[Dur10]{Durrett}
Rick Durrett, \emph{Probability: theory and examples}, fourth ed., Cambridge
  Series in Statistical and Probabilistic Mathematics, vol.~31, Cambridge
  University Press, Cambridge, 2010. \MR{2722836}

\bibitem[DVJ03]{VereJones}
D.~J. Daley and D.~Vere-Jones, \emph{An introduction to the theory of point
  processes. {V}ol. {I}}, second ed., Probability and its Applications (New
  York), Springer-Verlag, New York, 2003, Elementary theory and methods.
  \MR{1950431 (2004c:60001)}

\bibitem[DW78]{dGW}
David~L. DeGeorge and Nolan~R. Wallach, \emph{Limit formulas for multiplicities
  in $l^2(\gamma \setminus g)$}, Annals of Mathematics \textbf{107} (1978),
  no.~2, pp. 133--150 (English).

\bibitem[Ebe96]{Eberlein}
Patrick~B. Eberlein, \emph{Geometry of nonpositively curved manifolds}, Chicago
  Lectures in Mathematics, University of Chicago Press, Chicago, IL, 1996.
  \MR{1441541 (98h:53002)}

\bibitem[G{\etalchar{+}}10]{Goldman}
Andr{\'e} Goldman et~al., \emph{The palm measure and the voronoi tessellation
  for the ginibre process}, The Annals of Applied Probability \textbf{20}
  (2010), no.~1, 90--128.

\bibitem[Gel14]{Gelander}
Tsachik Gelander, \emph{Lectures on lattices and locally symmetric spaces},
  arXiv preprint arXiv:1402.0962 (2014).

\bibitem[Hel01]{Helgason}
Sigurdur Helgason, \emph{Differential geometry, {L}ie groups, and symmetric
  spaces}, Graduate Studies in Mathematics, vol.~34, American Mathematical
  Society, Providence, RI, 2001, Corrected reprint of the 1978 original.
  \MR{1834454 (2002b:53081)}

\bibitem[HKPV09]{HHPV}
J.~Ben Hough, Manjunath Krishnapur, Yuval Peres, and B\'alint Vir\'ag,
  \emph{Zeros of {G}aussian analytic functions and determinantal point
  processes}, University Lecture Series, vol.~51, American Mathematical
  Society, Providence, RI, 2009. \MR{2552864}

\bibitem[Kal86]{Kallenberg}
Olav Kallenberg, \emph{Random measures}, fourth ed., Akademie-Verlag, Berlin;
  Academic Press, Inc., London, 1986. \MR{854102}

\bibitem[Kni97]{Knieper}
Gerhard Knieper, \emph{On the asymptotic geometry of nonpositively curved
  manifolds}, Geometric \& Functional Analysis GAFA \textbf{7} (1997), no.~4,
  755--782.

\bibitem[Lee09]{Lee}
Jeffrey~M. Lee, \emph{Manifolds and differential geometry}, Graduate Studies in
  Mathematics, vol. 107, American Mathematical Society, Providence, RI, 2009.
  \MR{2572292}

\end{thebibliography}



\end{document}